\documentclass[11pt, a4paper]{amsart}
\usepackage[utf8]{inputenc}
\usepackage[english]{babel}
\usepackage{enumerate}
\usepackage{color}
\usepackage{stmaryrd}
\usepackage[text={15cm,21.5cm},centering]{geometry}
\usepackage{marginnote}
 \usepackage{framed}
\usepackage{siunitx}
\usepackage{xcolor}
\usepackage{esint,amsmath,amssymb}
\usepackage{graphicx}
\usepackage[showonlyrefs]{mathtools}
\usepackage[colorlinks=true, pdfstartview=FitV, linkcolor=blue, citecolor=blue, urlcolor=blue,pagebackref=false]{hyperref}
\usepackage{microtype}

\usepackage{bm}
\usepackage{scalerel}
\usepackage{mathrsfs}
\usepackage{tikz}
\usepackage[all]{xy} \xyoption{arc} \xyoption{color}
\usepackage{epsfig}
\usepackage{amsbsy}
\usepackage[font={footnotesize}]{caption}

\usepackage[skip=0cm,list=true,labelfont=it]{subcaption}
\usepackage{enumitem}

\usepackage{comment}

\newcommand{\eps}{\varepsilon}
\newtheorem{proposition}{Proposition}
\newtheorem{theorem}[proposition]{Theorem}

\theoremstyle{remark}
\newtheorem{remark}[proposition]{Remark}

\theoremstyle{definition}
\newtheorem{definition}[proposition]{Definition}

\newtheorem{assumption}[proposition]{Assumption}

\numberwithin{equation}{section}
\numberwithin{proposition}{section}

\title[Shallow water models for the dynamics of OWCs]{
Shallow water models for the dynamics of oscillating water columns
}

\author{Edoardo Bocchi}
	\address[Edoardo Bocchi]{Dipartimento di Matematica, Politecnico di Milano, Piazza Leonardo da
Vinci 32, 20133 Milano, Italy}
	\email{edoardo.bocchi@polimi.it}

\begin{document}

	\begin{abstract}
 We investigate the dynamics of an oscillating water column, which is a wave energy converter where water waves interact with a fixed partially immersed structure and an air chamber. Considering the shallow water regime, the fluid motion is governed by either the one-dimensional nonlinear shallow water equations or the Boussinesq-Abbott equations, while the air pressure variation in the chamber acts as a spring force on the fluid. The presence of the structure and the air chamber introduces constraints into the fluid equations. Assuming conservation of the total fluid-elastic energy in the absence of structural damping, we reformulate the constrained equations as transmission problems between the open water and the chamber and show their local well-posedness. We also consider a different modeling approach in which the upper part of the fluid in the chamber is treated as a rigid layer (water column) free to move vertically above the lower fluid. For this configuration, we reformulate the dynamics as wave-spring-mass systems, where the motion of the water column is driven by the shallow water waves outside the chamber. We establish their local well-posedness and show that the effective buoyancy period of the water column motion is determined by the competition between the added mass and the stiffness of the spring force.
	\end{abstract}

	\maketitle

\section{Introduction}

The increasing energy demand of society over the last decades has shifted the focus of engineers toward renewable energy sources. Among the available options, marine energy is one of the few that is not significantly affected by unfavorable atmospheric conditions, such as the lack of sunlight or wind, thus making it a promising energy source. However, a major challenge in the development of wave energy converters is to maximize their efficiency, defined as the ratio of the captured energy to the incident wave energy. This issue has been extensively investigated through both numerical and experimental studies in ocean engineering \cite{Reza13, Reza17, Reza18}. In order to design and develop efficient devices, a deep understanding of wave-structure interactions is fundamental and can be achieved through the analysis of accurate mathematical models describing the underlying phenomena.

Among the different wave energy converters (see \cite{Baba18} for an overview), here we are interested in the so-called \emph{oscillating water columns} (OWCs) installed onshore (see Figure \ref{owc}). In these devices, incident waves arriving from offshore encounter a fixed partially immersed structure with vertical lateral walls and enter a partially closed chamber. A turbine is located at the top of the chamber and activated by the airflow generated by the fluid volume variation; consequently, a generator connected to the turbine transforms the mechanical energy of the turbine into electric energy. The name of the converter is due to the behavior of the water inside the chamber: macroscopically, it behaves as a flat column vertically oscillating inside the chamber. One example of this type of device is the Pico Power Plant \cite{Dima17} installed on Pico Island, Azores (Portugal).

\begin{figure}
	\includegraphics[scale=0.6]{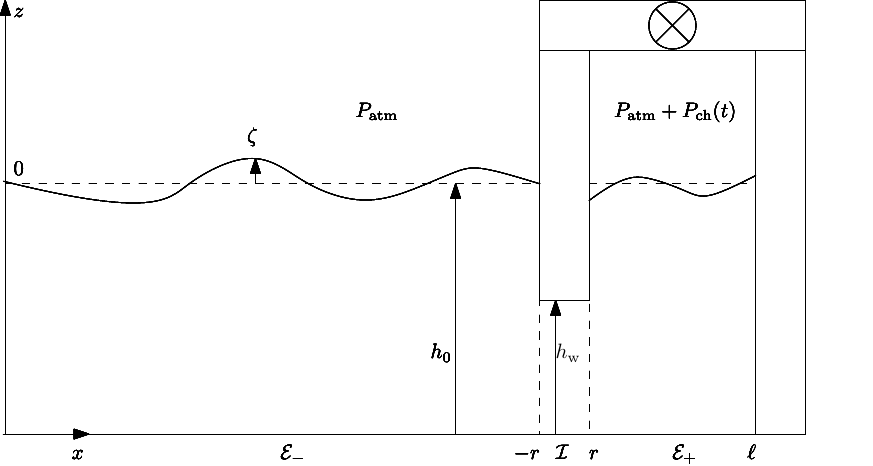}
	\caption{Configuration of the onshore OWC.}
	\label{owc}
\end{figure}

The equations governing the motion of ocean waves are the water waves equations, \textit{i.e.}, the free surface incompressible and irrotational gravity-driven Euler equations (see the monograph \cite{LanBook}). However, in most applications, they are too computationally expensive for direct numerical simulation. For this reason, instead of the full equations, shallow water asymptotic models are used to describe the dynamics in coastal zones. Here we consider two of these models, which can be unified in the one-dimensional system
\begin{equation}\label{intro-system}
	\begin{cases}
		\partial_t \zeta + \partial_x q=0 \\[5pt]
		\Big(1-\dfrac{\mu}{3}\partial_x^2\Big) \partial_t q + \varepsilon \partial_x \Big(\dfrac{q^2}{h}\Big) + h \partial_x \zeta=  -\dfrac{h}{\varepsilon}\partial_x{\underline{P}}
	\end{cases} \quad\text{with} \quad \mu\geq 0,
\end{equation}
where $\zeta$ is the fluid surface elevation,  $h$ is the fluid height, $q$ is the fluid horizontal discharge, and $\underline{P}$ is the fluid surface pressure deviation from atmospheric pressure (see Section \ref{sec-shallow} for more details). When the shallowness parameter $\mu=0$, \eqref{intro-system} corresponds to the nonlinear shallow water (or Saint-Venant) equations; when $\mu \neq 0$, it corresponds to the Boussinesq-Abbott equations. The advantage of these approximated models is twofold: the dimension of the full problem is reduced, as they are set in the horizontal projection of the fluid domain, and the domain is time-independent \cite{Lan20}.\\
In this study, we consider the shallow water models \eqref{intro-system} in the presence of an onshore OWC. This problem belongs to the broad class of wave-structure interaction problems, in which the fluid surface is in contact with a fixed or moving partially immersed structure or with a solid boundary.
 In the last century, the dynamics of floating objects was investigated from a theoretical point of view using simplified linear and time-harmonic models \cite{John1, John2, Ursell}. 
Considering water waves equations, a nonlinear depth-averaged model was introduced in \cite{Lan17} with a contact constraint accounting for the presence of the object.  Recently, well-posedness of the two-dimensional linearized equations around the rest state was established in the two-dimensional case for both fixed \cite{LanMin26} and moving objects \cite{LanPau25}. A major issue to deal with is the contact between the fluid surface and the solid boundary, as the fluid velocity becomes singular in the vicinity of contact points. It was previously studied using the full nonlinear equations in the shoreline problem, where waves approach a sloping beach, for which a priori estimates \cite{Poyferre19} and local-in-time well-posedness considering surface tension \cite{MinWan20, MinWan21, MinWan24} were obtained under smallness assumptions on the contact angles. Considering other fluid dynamics, the contact issue has been addressed taking into account both viscosity and capillarity: in the cases of Stokes \cite{ZheTice17,GuoTice18}, Navier-Stokes \cite{GuoTice24, GTWZ24} and Darcy (or one-phase Muskat) \cite{BocCasGan26} flows in a vessel with vertical lateral walls, no smallness assumptions are required.

Regarding the interaction problem with floating objects, well-posedness of shallow water models \eqref{intro-system} is now established for different dimensions and configurations. In the horizontal one-dimensional case, it was obtained for $\mu=0$ in \cite{IguLan21} as an application of a general theory on hyperbolic initial boundary value problems; for $\mu \neq 0$, it was first established in \cite{BreLanMet21} for a system formally equivalent to \eqref{intro-system} with a fixed object, and then in \cite{BecLan22} with a vertically moving object, including the study of the return to equilibrium scenario. In \cite{MS-MTT19, VerMatTuc21}, this approach was used for a viscous version of \eqref{intro-system} with $\mu=0$, while other techniques were employed in \cite{God-al-18, God-al-20} for numerical simulations of the problem.
In the two-dimensional case, the vertical motion and return to equilibrium of a structure with vertical walls were studied in the axisymmetric without swirl setting for $\mu=0$ in \cite{Boc20-1, Boc20-2} and recently for $\mu \neq 0$ in \cite{Bec-al-26}. The general two-dimensional setting was considered for fixed structures, first with vertical walls \cite{IguLan25} and later with non-vertical walls \cite{IguLan26}, where a free boundary problem due to the moving contact line must be addressed.\\
OWC devices have been intensively studied in ocean engineering using linear potential models \cite{Eva78, Eva82, EvaPor95, Reza15}. Following the nonlinear approach used for floating objects, the dynamics of OWCs in shallow water was modeled in \cite{BocHeVer21} considering \eqref{intro-system} with $\mu=0$ and a constant air pressure. Subsequently, the model was refined  in \cite{BocHeVer23}, to account for a more realistic time-dependent air pressure inside the chamber.  Assuming conservation of the total energy of the system, the  equations were reformulated as a transmission problem between the open water and chamber domains and its local well-posedness was established. In \cite{Par23}, the same fluid equations were coupled with an adiabatic law to model and numerically simulate trapped air pockets in coastal caves.\\
The aim of this study is threefold:
\begin{enumerate}
	\item To rigorously justify the notion of total energy introduced in \cite{BocHeVer23} that allows one to reformulate \eqref{intro-system} in the presence of an OWC as a transmission problem;\smallskip
	\item To take into account also dispersive effects in the dynamics of OWCs considering \eqref{intro-system} with $\mu \neq 0$;\smallskip
	\item To derive shallow water models for the dynamics of OWCs treating the upper part of the fluid in the chamber as a rigid layer free to move above the lower fluid.
\end{enumerate}
The main results corresponding to the previous objectives are:
\begin{enumerate}
	\item The introduction of the total energy in Definition \ref{def-totalenergy} and the derivation of coupling conditions in Proposition \ref{prop-intpres-con} ensuring total energy conservation; \smallskip
	\item The reformulation of \eqref{intro-system} with $\mu\neq0$ in the presence of an OWC as a transmission problem in Propositions \ref{prop-transmission} and \ref{prop-evo-qi} and its local well-posedness in Theorem \ref{theo-trans-wp-disp}; \smallskip
	\item The derivation of wave-spring-mass systems in Proposition \ref{prop-wavespringmass} and their local well-posedness for $\mu=0$ and $\mu\neq0$ respectively in Theorems \ref{theo-wave-spring-0} and \ref{theo-wave-spring}.
\end{enumerate}

\subsection{Outline of the paper}
In Section \ref{sec-shallow}, we introduce the asymptotic shallow water models that govern the fluid dynamics, the hyperbolic nonlinear shallow water equations and the dispersive Boussinesq-Abbott equations, describing the different physical regimes considered. Moreover, we recall the energy conservation at order $O(\eps\mu)$ exhibited by the two reduced models in the absence of structures: exact conservation in the hyperbolic case ($\mu=0$) and almost conservation in the dispersive case ($\mu\neq0$).\\ In Section \ref{sec-airpressure} we introduce the air pressure dynamics inside the chamber. First, we discuss the idealized scenario where the chamber is completely closed and the motion of air is a polytropic adiabatic process. Second, in a real scenario where structural damping occurs at the turbine, we introduce a linear ODE coupling the air pressure to the fluid mean surface elevation in the chamber.\\ In Section \ref{sec-constramodels}, we show how the presence of both the partially immersed structure and the air chamber introduces constraints into the fluid equations, thereby splitting the dynamics between the region under the structure and its complement. In order to couple back the separate regions, continuity of the discharge at the structure walls and conservation of the total energy at order $O(\eps\mu)$ where no damping occurs are assumed. The total energy is defined as the sum of the shallow water fluid energy and the elastic potential energy associated with the spring force induced by the air pressure deviation from atmospheric pressure inside the chamber. As a conditional result for total energy conservation, coupling conditions are obtained in the form of shallow water Bernoulli principles. The constrained shallow water models are then reformulated as transmission problems between the open water and the chamber coupled with the evolution equation of the fluid discharge under the structure, which manifests an additional leading-order term in the dispersive case. Local well-posedness of the transmission problems is then shown. In the hyperbolic case, we apply the result in \cite{BocHeVer23}, where the associated initial boundary value problem is solved by means of a Kreiss symmetrizer that provides a hidden trace regularity crucial for the closure of a priori estimates. In the dispersive case, the problem is reduced to an infinite-dimensional ODE and Cauchy-Lipschitz theory applies.\\ In Section \ref{sec-springmass}, we introduce a different modeling approach by considering the fluid in the chamber as a multiphase system: the upper part is treated as a rigid layer (water column) free to move vertically above the lower part. The fluid equations in the chamber domain are then written in terms of the interface between the two phases, for which a new contact constraint is imposed, and of the interface pressure, which becomes an unknown of the problem as it depends on the motion of the water column. This is determined by a second order integro-differential equation on the fluid mean surface elevation, with restoring terms due to hydrostatic and spring forces, coupled with the fluid equations in the chamber. Following the analysis carried out in Section \ref{sec-constramodels}, we assume  conservation of the sum of the fluid energy and the mechanical energy of the water column to derive coupling conditions in the form of shallow water Bernoulli principles. These allow us to reformulate the dynamics of OWCs as wave-spring-mass systems, where the motion of the water column is driven by the shallow water waves outside the chamber. Considering both fluid models, an added mass phenomenon appears in the integro-differential equation and an additional contribution is due to dispersive effects, which also change the nature of the phenomenon introducing a new coupling with the fluid dynamics outside the chamber. An interesting fact is that the effective buoyancy period of the water column motion is determined by the competition between the added mass and the stiffness of the spring force, for which three scenarios are possible. Finally, analogously to Section \ref{sec-constramodels}, local well-posedness of the wave-spring-mass systems is then shown in both hyperbolic and dispersive cases.\\
In Appendix \ref{sec-appendix}, we show the details of the non-dimensionalization of the equations and the elastic potential energy introduced in the paper.
\subsection{General notation}Throughout the paper, for a function $f$ defined on the spatial domain $(-\infty,\ell)$, we denote by $f_{|_{D}}$ the restriction of $f$ to a subset $D\subset (-\infty, \ell)$.
Whenever no confusion may arise, we use for simplicity the notation $f(r)$ for the spatial trace $f(t,r)$ of a function $f(t,x)$, with the time dependence omitted. Similarly, the notation $f=g$ at $x=r$ stands for the equality $f(t,r)=g(t,r)$.

\section{Shallow water models}\label{sec-shallow}
In this section, we discuss the equations governing the fluid dynamics of OWC devices. We consider the two-dimensional motion of a homogeneous, inviscid, incompressible and irrotational fluid in the shallow water regime. More precisely, we address the physical scenario where the characteristic vertical scale of the problem, the still water depth $h_0$, is much smaller than the characteristic horizontal scale, the wavelength $L$. After defining the shallowness parameter by
$$\mu=\frac{h_0^2}{L^2}\geq 0,$$ 
we say that the \emph{shallow water regime} holds when $\mu\ll 1$. In this study, we consider shallow water waves of either large or small amplitude. More precisely, after introducing the nonlinearity parameter  
$$\eps=\frac{a}{h_0}\in (0,1],$$ 
where $a$ denotes the characteristic wave amplitude, we refer to the case $\eps=O(1)$ as the \emph{large amplitude regime} and to the case $\eps=O(\mu)$ as the \emph{small amplitude (or weakly nonlinear) regime}. \\
In the shallow water-large amplitude regime, the water waves equations are approximated with precision $O(\mu)$ by the nonlinear shallow water equations
\begin{equation}\label{NSW}
	\begin{cases}
		\partial_t \zeta + \partial_x q=0 \\[5pt]
		\partial_t q + \partial_x \Big(\dfrac{q^2}{h}\Big) + gh \partial_x \zeta= -\dfrac{h}{\rho}\partial_x{\underline{P}}.
	\end{cases}
\end{equation}
These equations are set in the horizontal projection of the time-dependent fluid domain and the unknowns are the fluid surface elevation $\zeta$, assumed to be the graph of a function of $x\in \mathbb{R}$, the fluid discharge $q$ defined as the horizontal fluid velocity vertically integrated over the fluid height and the fluid surface pressure $P_{\rm atm} + \underline{P}$, where $P_{\rm atm}$ is the reference constant atmospheric pressure. Here, $h=h_0+\zeta$ is the fluid height, $g$ is the gravity constant and $\rho$ is the (constant) fluid density. The first equation is called the continuity equation and it accounts for mass conservation, while the second equation is called the momentum equation and it accounts for conservation of the linear momentum. Analytically speaking, \eqref{NSW} represents a hyperbolic system \cite{BenSerBook}.\\
In the shallow water-small amplitude regime, the water waves equations are approximated with precision $O(\mu^2)$ by the Boussinesq-Abbott equations
\begin{equation}\label{Bou-Abb}
	\begin{cases}
		\partial_t \zeta + \partial_x q=0 \\[5pt]
		\Big(1-\dfrac{h_0^2}{3}\partial_x^2\Big) \partial_t q + \partial_x \Big(\dfrac{q^2}{h}\Big) + gh \partial_x \zeta= -\dfrac{h}{\rho}\partial_x{\underline{P}}.
	\end{cases}
\end{equation}The additional second-order term in the momentum equation of \eqref{Bou-Abb} accounts for dispersion in the dynamics, hence \eqref{Bou-Abb} can be seen as a dispersive perturbation of the hyperbolic system \eqref{NSW}. Many other Boussinesq-type systems formally equivalent to \eqref{Bou-Abb} at precision $O(\mu^2)$ can be derived from the water waves equations (see \cite{LanBook}), but the interest of $\eqref{Bou-Abb}$ lies in its perturbative structure. 

Writing the equations with dimensionless unknowns and variables (see \cite[Appendix A]{BecLan22} for the detailed non-dimensionalization), the previous shallow water models can be unified under the system
\begin{equation}\label{dimless-shallow}
	\begin{cases}
		\partial_t \zeta + \partial_x q=0 \\[5pt]
		\Big(1-\dfrac{\mu}{3}\partial_x^2\Big) \partial_t q + \eps \partial_x \Big(\dfrac{q^2}{h}\Big) + h \partial_x \zeta=  -\dfrac{h}{\eps}\partial_x{\underline{P}}\\[5pt]
	\end{cases} \quad\text{with} \quad \mu\geq 0,
\end{equation}
where the dimensionless fluid height is $h=1+\eps \zeta$.
The dimensionless nonlinear shallow water equations correspond to $\mu =0$ and the dimensionless Boussinesq-Abbott equations to $\mu\neq 0$. 
Another difference between the hyperbolic and the dispersive case is the nature of the energy conservation they exhibit. Indeed, after multiplying the continuity equation in \eqref{dimless-shallow} by $\zeta$, the momentum equation by $q/h$ and adding them up, we find the local energy balance 
\begin{equation}\label{energy-balance}
	\partial_t e + \partial_x f = \frac{\underline{P} \partial_x q  }{\eps}+\eps\mu R,
\end{equation}
where the fluid energy density  \begin{equation}\label{energy-density}
	e(\zeta, q)=e_{\rm gra}(\zeta)+ e_{\rm kin}(\zeta, q)
\end{equation} is the sum of the gravitational potential and kinetic energy densities
\begin{equation*}
	e_{\rm gra}(\zeta) =\frac{\zeta^2}{2}, \qquad  e_{\rm kin}(\zeta, q)= \frac{q^2}{2h} + \mu \frac{(\partial_x q)^2}{6h},
\end{equation*} the energy flux is
\begin{equation}\label{flux}
	f(\zeta, q, \underline{P})= q\Big( \zeta + \frac{\underline{P}}{\eps}+ \eps\frac{q^2}{2h^2} -\mu\frac{\partial_x \partial_t q}{3h}  \Big)
\end{equation}and the remainder term is 
\begin{equation}\label{R}
	R(\zeta, q)=\frac{(\partial_x q)^3}{6h^2}+\frac{q\partial_x \zeta\partial_t \partial_x q}{3h^2}.
\end{equation}
Under the usual assumption of constant surface pressure for water waves, we have $\underline{P}=0$. Then, the nonlinear shallow water equations ($\mu=0$) exhibit \emph{exact} local energy conservation, 
while the Boussinesq-Abbott equations ($\mu\neq 0$) \emph{almost} conserve energy, in the sense that conservation holds up to terms of order $O(\eps\mu)$, 
which in the weakly nonlinear regime $\eps=O(\mu)$ have the same order $O(\mu^2)$ as the precision of the approximated model. Note that some terms in the energy density and flux can be replaced by equivalent terms up to remainder terms of order $O(\eps^2, \eps\mu)=O(\mu^2)$ in the weakly nonlinear regime, for instance
$$\mu \frac{(\partial_x q)^2}{6h}= \mu \frac{(\partial_x q)^2}{6} + O(\eps \mu),$$ without modifying the structure of the energy balance. Nevertheless, here we work with \eqref{energy-balance} as it is an exact equation derived from \eqref{dimless-shallow}. We refer to \cite{BreLanMet21} for a different Boussinesq system, formally equivalent to \eqref{dimless-shallow}, which exactly conserves a different energy that, unlike $e(\zeta, q)$, is not the asymptotic expansion of the energy of the full water waves equations (see \cite[Section 6.3.1]{LanBook}).\\
If \eqref{dimless-shallow} is set in $(a,b)\subseteq \mathbb{R}$ and $f$ is a piecewise continuous function, integrating \eqref{energy-balance} over $(a,b)$ yields the global energy balance 
\begin{equation}\label{glo-energybal}
	\frac{dE}{dt}  +  f(b^-)- f(a^+) - \sum_{i=1}^{N-1} (f(x_i^+)- f(x^-_{i})) = \int_a^b \frac{\underline P \partial_x q}{\eps} + \eps\mu  \int_a^b R,
\end{equation}where \begin{equation}\label{fluid-energy}E(\zeta, q)=E_{\rm gra}(\zeta) + E_{\rm kin }(\zeta, q)= \int_a^b e_{\rm gra}(\zeta) + \int_a^b e_{\rm kin}(\zeta, q)\end{equation} is the total fluid energy, $x_1,\dots, x_{N-1}$ are the (possible) discontinuity points of $f$ and $f(c^\pm)= \lim_{x\rightarrow c^{\pm}}f(x)$. For regular solutions and in the absence of partially immersed objects, $f$ is continuous in $(a,b)$ and the surface pressure deviation $\underline{P}=0$. Thus, \eqref{glo-energybal} yields
\begin{equation}\label{energy-conservation}
	\frac{dE}{dt} = \eps\mu  \int_a^b R,
\end{equation}showing energy conservation at order $O(\eps\mu)$ for \eqref{dimless-shallow}: exact energy conservation for the nonlinear shallow water equations and almost energy conservation for the Boussinesq-Abbott equations. 
We will see in Section \ref{sec-constramodels} that this is no longer the case when considering \eqref{dimless-shallow} in the presence of an OWC device. This is due to two main facts: 
\begin{itemize}
	\item The energy flux ceases to be continuous at the walls of the immersed structure;\smallskip
	
	\item The surface pressure deviation does not vanish inside the chamber.
\end{itemize}
The first property is due to the vertical orientation of the walls, where contact with waves occurs, and it also manifests for other partially immersed structures, such as heaving buoys. 
The second property is instead specific to the OWC device, as it is related to the dynamics of the air partially trapped inside the chamber and whose pressure deviates from the atmospheric pressure as waves enter the chamber. Before analyzing the wave-structure interaction problem, the next section is devoted to the modeling of the air pressure in the chamber and, in particular, to its coupling with the fluid beneath it.

\section{Air pressure dynamics in the chamber}\label{sec-airpressure}
We discuss here the air pressure inside the OWC chamber, which from now on is assumed to be uniform in space. Let us first consider the case in which the chamber is completely closed at the top and the air is trapped inside. The reason why we study this idealized scenario will be explained later in Section \ref{sec-constramodels}, where we reformulate the constrained shallow water models as transmission problems between the open water and the chamber. In this scenario, we consider the air motion as a generic polytropic process, hence the air pressure and volume satisfy
\begin{equation}\label{polytropic}
	\frac{d}{dt} \left(P_{\rm air}V^\gamma_{\rm air}\right)= 0,
\end{equation}
where $\gamma>0$ is the polytropic index of the air. Following \cite{Par23}, we assume that the relaxation time of the air pressure is negligible, so that the deformation process is fast and there is no heat transfer. For this type of adiabatic process, $\gamma$ is equal to the heat capacity ratio ($\gamma = 1.4$ for air). At rest, the air pressure is equal to the atmospheric pressure $P_{\rm atm}$ and the air volume is $V_0= |\mathcal{E}_+| h_{\rm ch}$, where $h_{\rm ch}>0$ is the height of the region inside the chamber filled by air at rest (called ``chamber height'' in \cite[Section 2]{BocHeVer23}). We then obtain from \eqref{polytropic} the adiabatic law
\begin{equation}\label{adiabatic}
	P_{\rm air}(t)V^\gamma_{\rm air}(t)= P_{\rm atm }|\mathcal{E}_+|^\gamma h_{\rm ch}^\gamma.
\end{equation}
Denoting by $\overline{\zeta}$ the fluid mean surface elevation in the chamber, 
we have that $V_{\rm air}(t)= |\mathcal{E}_+|\left(h_{\rm ch}-\overline{\zeta}(t) \right) $ and  
\eqref{adiabatic} can be written in terms of $\overline{\zeta}$ as
\begin{equation}\label{airpres-adiabatic}
	P_{\rm air}(t)= P_{\rm atm} \Big( \frac{h_{\rm ch}}{h_{\rm ch} - \overline{\zeta}(t)}\Big)^\gamma.
\end{equation}Assuming that the oscillations of the mean surface elevation are small compared to the chamber height, that is, $\overline{\zeta}\ll h_{\rm ch}$, \eqref{airpres-adiabatic} can be linearly approximated an the air pressure deviation from atmospheric pressure reads
\begin{equation}\label{Pch-closed}
	P_{\rm ch}(t)= P_{\rm air}(t)- P_{\rm atm} = \frac{\gamma P_{\rm atm }}{h_{\rm ch}}\overline{\zeta}(t).
\end{equation}

In a real OWC device, the chamber is partially closed at the top with opening width $b>0$, where the turbine is installed, and one has to take into account both the air chamber thermodynamics and the turbine aerodynamics. Considering the specific Wells turbines \cite{Rag95}, which are linear power take-off systems, and following \cite{Fal16,Dima17}, under the assumptions that $\overline{\zeta}\ll h_{\rm ch}$ and $P_{\rm ch} \ll P_{\rm atm}$, the evolution of the air pressure deviation is approximated by the first-order linear ODE 
\begin{equation}\label{Pch-dyn}
	\frac{d P_{\rm ch}}{dt} +k P_{\rm ch} = 	\frac{\gamma P_{\rm atm }}{h_{\rm ch}}\frac{d\overline{\zeta}}{dt},
\end{equation}
considered also in \cite{BocHeVer23}. Here, $k>0$ is a damping constant proportional to the ratio $b/h_{\rm ch}$ and to the inverse of the turbine rotational speed. After imposing the additional condition that at initial time $P_{\rm ch}$ and $\overline{\zeta}$ are related through \eqref{Pch-closed}, the solution to \eqref{Pch-dyn} is explicitly given by
\begin{equation}\label{Pch-dyn0}
	P_{\rm ch}(t) = \frac{\gamma P_{\rm atm }}{h_{\rm ch}}\left(\overline{\zeta}(t) - k \int_{0}^t e^{-k (t-s)}\overline{\zeta}(s)ds\right).
\end{equation} 	Note that the case when $k$ is negligible, that is, when structural damping at the turbine power take-off is negligible (called \emph{non-damped} scenario in \cite{BocHeVer23}), is mathematically equivalent to the case when the chamber is completely closed ($b=0$), as \eqref{Pch-dyn0} reduces to \eqref{Pch-closed}.\\	
Passing to the dimensionless form (see Appendix \ref{appsec-air}), \eqref{Pch-dyn} reads 
\begin{equation}\label{evo-Pch-dimless}
	\dfrac{dP_{\rm ch}}{dt} + \kappa P_{\rm ch} = \epsilon \frac{d \overline{\zeta}}{dt}, \qquad P_{\rm ch}(0)=\epsilon \overline\zeta(0),
\end{equation}where $\kappa=\frac{kL}{\sqrt{gh_0}}$ is the dimensionless damping constant and 
\begin{equation}\label{epsilon}
	\epsilon=\frac{a}{h_{\rm ch}}\in (0,1]
\end{equation}
is the compression parameter, measuring the ratio og the characteristic wave amplitude to the chamber height. The unique solution to \eqref{evo-Pch-dimless} is given by
\begin{equation}\label{Pch-dimless}
	P_{\rm ch}(t)= \epsilon \left(I - \kappa\mathcal{D}\right)\overline{\zeta}(t),
\end{equation}where $I$ is the identity operator and $\mathcal{D}$ is the delay operator defined by
\begin{equation}\label{delay}
	\mathcal{D}\overline{\zeta}(t)= \int_{0}^t e^{-\kappa (t-s)}\overline{\zeta}(s) ds.
\end{equation}

\section{Constrained shallow water models}\label{sec-constramodels}
In this section, we derive nonlinear models for the dynamics of OWCs starting from the shallow water models presented in Section \ref{sec-shallow}. Dealing with onshore devices interacting with waves coming from offshore, we consider \eqref{dimless-shallow} in the horizontal domain $(-\infty, \ell)$, with $x=\ell>0$ being the horizontal position of the shoreline, that is,
\begin{equation}\label{dimless-shallow-set}
	\begin{cases}
		\partial_t \zeta + \partial_x q=0 \\[5pt]
		\Big(1-\dfrac{\mu}{3}\partial_x^2\Big) \partial_t q + \eps \partial_x \Big(\dfrac{q^2}{h}\Big) + h \partial_x \zeta=  -\dfrac{h}{\eps}\partial_x{\underline{P}}\\[5pt]
	\end{cases} \quad\text{in} \quad(-\infty, \ell)
\end{equation}
complemented with the rest far-field condition and wall condition at the shoreline
\begin{equation}\label{bound-cond}
	(\zeta, q)\rightarrow(0,0) \quad \text{as}\quad x\rightarrow -\infty \qquad \text{and}\qquad q=0 \quad \text{at} \quad x=\ell.
\end{equation}
The OWC chamber is located at the right endpoint of the domain and its entrance is preceded by a partially immersed structure with vertical walls located at $x=\pm r$, with $0<r< \ell$. The presence of the object makes $(-\infty, \ell)$ naturally split into two disjoint subdomains: the \emph{interior domain} $\mathcal{I}=(-r,r)$, which is the horizontal projection of the fluid region under the structure, and the \emph{exterior domain} $\mathcal{E}= (-\infty, \ell)\setminus \mathcal{I}$. We refer to the intersection $\overline{\mathcal{E}}\cap\overline{\mathcal{I}}= \{ \pm r\}$ as the \emph{contact points}: rigorously speaking, they are the horizontal projections of the points where waves touch the walls of the structure. In addition, the exterior domain has two connected components: the \emph{open water domain} $\mathcal{E}_-=(-\infty, -r)$ and the \emph{chamber domain} $\mathcal{E}_+=(r, \ell)$.\\
It is well-known \cite{Lan17} that the interaction between waves and partially immersed structures can be mathematically described by fluid models with a contact constraint on the surface elevation, which must match the profile of the solid bottom. This is an example of a congestion phenomenon in fluid models (see \cite{Per18}). Assuming that the solid bottom is parametrized by the dimensionless function $\zeta_{\rm w}(x)$, we then impose the constraint
\begin{equation}\label{contact-con}
	\zeta=\zeta_{\rm w} \quad \mbox{in} \quad \mathcal{I}.
\end{equation}
On the contrary, the surface pressure $\underline{P}$ beneath the structure becomes an unknown of the problem and it plays the role of a Lagrange multiplier associated with \eqref{contact-con}. Since the structure is fixed, we have that $\partial_t \zeta_{\rm w}=0$ and the equations in the interior domain read
\begin{equation*}\label{int-dynamics}
	\begin{cases}
		q=q_i(t) \\[5pt]
		\dfrac{dq_i}{dt} - \eps \dfrac{q_i^2}{h_{\rm w}^2}\dfrac{dh_{\rm w}}{dx} + h_{\rm w} \dfrac{d \zeta_{\rm w}}{dx}=  -\dfrac{h_{\rm w}}{\eps}\partial_x{\underline{P}}\\[5pt]
	\end{cases} \quad \text{in} \quad \mathcal{I}.
\end{equation*}
Note that the dispersive contribution in the momentum equation disappears as the interior discharge depends only on time. For simplicity, we consider a partially immersed structure with a flat bottom, so that the previous system reduces to 
\begin{equation}\label{NSW-interior}
	\begin{cases}
		q=q_i(t) \\[5pt]
		\dfrac{dq_i}{dt}=  -\dfrac{h_{\rm w}}{\eps}\partial_x{\underline{P}}\\[5pt]
	\end{cases} \quad \text{in} \quad \mathcal{I}.
\end{equation}
In the exterior domain, it is usually assumed that the surface pressure matches the air pressure. For boat or buoy configurations, the air above the fluid is not trapped and there is no deviation from atmospheric pressure, that is, $\underline{P}=0$ in $\mathcal{E}$. In contrast, for OWC devices, we must distinguish between the open water and chamber domains. We then impose that 
\begin{equation}\label{surpres-ext}
	\underline{P}=0 \quad \text{in} \quad \mathcal{E}_- \qquad \text{and}\qquad \underline{P}=P_{\rm ch}\quad \text{in} \quad \mathcal{E}_+,
\end{equation}
where $P_{\rm ch}$ is the air pressure deviation discussed in Section \ref{sec-airpressure}. 
The non-zero pressure constraint in the chamber introduces additional congestion in the modeling but, since the air pressure is assumed to be uniform in space, the fluid equations are not directly affected by the new constraint, as the source term in the momentum equation of \eqref{dimless-shallow-set} vanishes in both $\mathcal{E}_\pm$, like for boat/buoy configurations. 
Nevertheless, it will play a key role in the derivation of a coupling condition between the interior and the exterior equations allowing us to reformulate \eqref{dimless-shallow-set}, \eqref{contact-con} and  \eqref{surpres-ext} as transmission problems between $\mathcal{E}_-$ and $\mathcal{E}_+$ across the contact points $x=\pm r$.\\ 
After writing \eqref{surpres-ext} in dimensionless form (see Appendix \ref{appsec-air}) and combining it with \eqref{Pch-dimless}, we find 
\begin{equation}\label{pressure-mean}
	\underline{P}=0 \quad \text{in} \quad \mathcal{E}_-, \qquad
	\underline{P}=\eps P_0 \left(I -\kappa \mathcal{D} \right)\overline{\zeta}\quad \text{in} \quad \mathcal{E}_+
\end{equation}with $P_0= \frac{\gamma P_{\rm atm}}{\rho g h_{\rm ch}}$, where the fluid mean surface elevation in the chamber is defined by 
 \begin{equation}\label{mean-sur} \overline{\zeta}=\frac{1}{|\mathcal{E}_+|}\int_{\mathcal{E}_+}\zeta(x)dx, \end{equation}
and in the second equality we have used the relation $\frac{\epsilon}{h_0}= \frac{\eps}{h_{\rm ch}}$. After injecting \eqref{pressure-mean} into \eqref{dimless-shallow-set}, the equations in the exterior domain read
\begin{equation}\label{NSW-exterior}
	\begin{cases}
		\partial_t \zeta + \partial_x q=0 \\[5pt]
		\Big(1-\dfrac{\mu}{3}\partial_x^2\Big) \partial_t q + \eps \partial_x \Big(\dfrac{q^2}{h}\Big) + h \partial_x \zeta=0
	\end{cases}\quad \text{in}\quad \mathcal{E}.
\end{equation} The presence of an OWC introduces a \emph{compressible-incompressible} transition between the exterior and interior domains, in the sense that \eqref{NSW-exterior} (with $\mu=0$) has the same form as the isentropic compressible Euler equations, whereas \eqref{NSW-interior} must be treated as an incompressible system associated with the constraint \eqref{contact-con}. In fact, this dichotomy holds at the level of the full water waves equations \cite{Lan17} and is inherited by the shallow water reduced models.

\subsection{Coupled problems between the exterior and interior domains}
So far, starting from \eqref{dimless-shallow-set}, we have derived two systems of equations, \eqref{NSW-interior} and \eqref{NSW-exterior}, that separately govern the fluid dynamics in the interior and exterior domains. In order to derive a closed formulation of the original problem, both systems must be complemented with boundary conditions, which necessarily lead to coupling conditions at the contact points. To this end, in the spirit of \cite{BecLan22}, we introduce the following properties:
\smallskip
\begin{enumerate}
	\item[(i)] The fluid discharge is continuous at the contact points;\smallskip
	\item[(ii)] When the chamber is closed (or in the non-damped scenario), the total energy of the system is conserved at order $O(\eps\mu)$: that is, it is exactly conserved when $\mu=0$ and almost conserved when $\mu\neq 0$.\smallskip
\end{enumerate}

Property (i) is deduced from the slip condition for the fluid velocity on the solid boundary that complements the water waves equations in the presence of a partially immersed object (see \cite{Lan17}).\\
Property (ii) is less direct and deserves further attention. In fact, in the case of a structure with non-vertical walls and under the assumption of constant air pressure, the continuity of the surface elevation at the contact points
\begin{equation}\label{surface-con}
	\zeta_{|_{\mathcal{E}}}= \zeta_{\rm w} \quad \text{at} \quad x=\pm r
\end{equation}
implies the continuity of the surface pressure
\begin{equation} \label{surpre-con}
	\underline{P}_{|_{\mathcal{I}}}= 0 \quad \text{at} \quad x=\pm r,
\end{equation} which allows one to fully determine the surface pressure in the interior domain from \eqref{NSW-interior}.
The continuity of the fluid surface, discharge and surface pressure then yields the continuity of the energy flux \eqref{flux} at the contact points. When the structure is fixed, the first term on the right-hand side of \eqref{glo-energybal} vanishes due to \eqref{NSW-interior} and \eqref{surpre-con} and we recover the conservation \eqref{energy-conservation}  of the fluid energy \eqref{fluid-energy} (with $(a,b)=(-\infty, \ell)$) like in the absence of structures. When the structure moves vertically \cite{Lan17}, the first term on the right-hand side of \eqref{glo-energybal} does not vanish, as it depends on the solid motion, and we obtain from \eqref{glo-energybal} that the conserved energy is the sum of the fluid and the solid mechanical energies.\\
In the case of a structure with vertical walls, the continuity of the fluid surface \eqref{surface-con} and surface pressure \eqref{surpre-con} at the contact points no longer holds. However, the principle of conservation of total energy remains valid and we assume it directly. Although the structure is fixed, the total energy of the system does not coincide with the fluid energy \eqref{fluid-energy}, contrary to the first approach developed in \cite{BocHeVer21}. Indeed, the external forces acting on the fluid are not limited to the gravitational force, but also include the force due to the air pressure variation in the chamber. Since we assume a spatially uniform air pressure, this pressure force reads
\begin{equation*}
	F_{\rm pres}= - \int_{\mathcal{E}_+} P_{\rm ch}= -|\mathcal{E}_+|P_{\rm ch}.
\end{equation*}
To identify the appropriate notion of total energy, we examine the case where the chamber is completely closed or, equivalently, the non-damped scenario, both representing a physical situation with no energy loss at the turbine power take-off. Obviously, in a real OWC device, part of the energy is lost and transformed into mechanical energy of the turbine. In the idealized scenario, \eqref{Pch-closed} or \eqref{Pch-dyn0} with $k=0$ hold and we derive the Hooke's law 
\begin{equation}\label{spring-force}
	F_{\rm pres}= -K \overline{\zeta} \qquad \text{with}\quad K= \frac{\gamma P_{\rm atm}|\mathcal{E}_+|}{h_{\rm ch}}.
\end{equation}
Thus, the pressure force acting on the fluid inside the chamber plays the role of a spring force with stiffness $K$ that tends to bring the mean surface elevation $\overline\zeta$ back to the zero equilibrium level. In this sense, the fluid inside the chamber can be interpreted as a liquid piston that moves up and down subject to both gravitational and spring forces. Associated with Hooke's law \eqref{spring-force}, we introduce the dimensionless elastic potential energy (see Appendix \ref{appsec-elastic})
\begin{equation}\label{ela-energy-dimless}
	E_{\rm ela}(\zeta)=\frac{\tau}{2} \overline{\zeta}^2,
\end{equation}
with dimensionless stiffness $\tau=P_0|\mathcal{E}_+|$. We stress that the elastic potential energy $E_{\rm ela}$ rigorously justifies the OWC energy $E_{\rm OWC}$ introduced in \cite{BocHeVer23}. Indeed, in the non-damped scenario, combining \eqref{Pch-dyn0} with the dimensional form of \eqref{ela-energy-dimless} (see Appendix \ref{appsec-elastic}) yields
$$E_{\rm ela}(\zeta)= \frac{1}{2} \frac{\gamma P_{\rm atm}|\mathcal{E}_+|}{h_{\rm ch}} \overline\zeta^2= \frac{|\mathcal{E}_+| h_{\rm ch}}{2\gamma P_{\rm atm}}P_{\rm ch}^2= \frac{1}{2\gamma_2}P_{\rm ch}^2= E_{\rm OWC}.$$
We are now ready to introduce the definition of total energy of the system:
\begin{definition}\label{def-totalenergy} 
	The \emph{total energy} of the constrained shallow water models \eqref{dimless-shallow-set}-\eqref{contact-con} and \eqref{pressure-mean} is
	\begin{equation}\label{total-energy}
		E_{\rm tot}(\zeta,q)=E(\zeta,q)+ E_{\rm ela }(\zeta),
	\end{equation}with $E$ and $E_{\rm ela}$ respectively as in \eqref{fluid-energy} and \eqref{ela-energy-dimless}.
\end{definition}

In view of \eqref{NSW-interior} and \eqref{total-energy}, the properties (i)-(ii) can be then  made explicit as follows:

\begin{assumption}\label{assumption}The following properties hold:
	\begin{equation}
		\tag{i}  \label{discharge-con}
		q_{|_{\mathcal{E}}} =q_i\quad \text{at} \quad x=\pm r;
	\end{equation}
	\begin{equation}
		\tag{ii} \label{tot-ene-conservation} 
		\text{when $\kappa=0$ in \eqref{pressure-mean},} \quad  
		\frac{d}{dt}E_{\rm tot} = \eps\mu \int_{-\infty}^\ell R = O(\eps\mu),
	\end{equation} with $R$ as in \eqref{R}. 
	
\end{assumption}

Our next goal is to derive an explicit coupling condition at the contact points from item \eqref{tot-ene-conservation} of Assumption \ref{assumption}. To this end, let us first introduce some notation. For a function $F$ defined in $(-\infty, \ell)$, we denote the jump of $F$ between $\mathcal{E}$ and $\mathcal{I}$ at the contact points $x= \pm r$ by $$\llbracket F\rrbracket^*_{\pm r}= F_{|_{\mathcal{E}}}(\pm r) - F_{|_{\mathcal{I}}}(\pm r). $$ For a function $G$ defined either in $\mathcal{I}$ or in $\mathcal{E}$, we denote the jump and (arithmetic) mean of $G$ across the contact points respectively by 
\begin{equation}\label{jump-mean}
	\llbracket G \rrbracket = G(r)- G(-r) \qquad \text{and}\qquad \langle G\rangle= \frac{G(r)+G(-r)}{2}.
\end{equation}In particular, we have the equivalence 
\begin{equation}\label{equiv-brackets}
	\left\llbracket F\right\rrbracket^*_{-r}=\left \llbracket F\right \rrbracket^*_{r} \ \iff \	\left\llbracket F_{|_{\mathcal{I}}}\right\rrbracket=  \left \llbracket F_{|_{\mathcal{E}}}\right\rrbracket.
\end{equation}

\begin{proposition}\label{prop-conserve-coup}
	Consider regular solutions $(\zeta, q)$ to the constrained shallow water models \eqref{dimless-shallow-set}-\eqref{contact-con} and \eqref{pressure-mean} with $\kappa=0$. Then, item \eqref{tot-ene-conservation} of Assumption \ref{assumption} holds if and only if
	\begin{equation}\label{f-con}
		\left\llbracket f\right\rrbracket^*_{-r}=\left \llbracket f\right \rrbracket^*_{r},
	\end{equation} with $f$ as in \eqref{flux}.
	In addition, using item \eqref{discharge-con} of Assumption \ref{assumption}, \eqref{f-con} reduces to
	\begin{equation}\label{bra-con}
		\left\llbracket \zeta + \frac{\underline{P}}{\eps} + \mathfrak{E}\right\rrbracket^*_{-r}  =\left \llbracket \zeta +\frac{\underline{P}}{\eps}+\mathfrak{E}\right \rrbracket^*_{r} 
	\end{equation}
	or, equivalently, 
	\begin{equation}\label{bra-con-equi}	\frac{\llbracket  \underline{P}_{|_{\mathcal{I}}}\rrbracket}{\eps}=  \left \llbracket \zeta_{|_{\mathcal{E}}} + \mathfrak{E}_{|_{\mathcal{E}}}+ \frac{\underline{P}_{|_{\mathcal{E}}}}{\eps}\right\rrbracket,
	\end{equation}	with the dimensionless shallow water dynamic pressure given by
	\begin{equation}\label{frak-E}
		\mathfrak{E}(\zeta, q)=  \eps  \frac{q^2}{2h^2}-\mu\frac{\partial_x\partial_t q}{3h}.
	\end{equation}
\end{proposition}
\begin{proof}
	Combining \eqref{glo-energybal} for $(a,b)=(-\infty, \ell)$ with \eqref{NSW-interior}, \eqref{pressure-mean}, \eqref{NSW-exterior} and \eqref{ela-energy-dimless}, and using the far-field and wall conditions \eqref{bound-cond}, we obtain 
	\begin{equation*}\begin{aligned}
			&\frac{dE}{dt}  -\llbracket f\rrbracket^*_{r}+ \llbracket f \rrbracket^*_{-r}=   -P_0(I-\kappa \mathcal{D}) \overline\zeta \int_{\mathcal{E}_+}\partial_t \zeta+ \eps \mu \int_{-\infty}^\ell R
			\\[5pt]&=-\frac{d}{dt}\Big(\frac{\tau}{2}\overline\zeta^2\Big) + \tau\kappa\mathcal{D}\overline\zeta \frac{d\overline\zeta}{dt} + \eps\mu \int_{-\infty}^\ell R= -\frac{dE_{\rm ela}}{dt}+ \tau\kappa\mathcal{D}\overline\zeta \frac{d\overline\zeta}{dt}  + \eps\mu \int_{-\infty}^\ell R.
		\end{aligned}
	\end{equation*} Using \eqref{total-energy}, we deduce that
	\begin{equation*}
		\frac{dE_{\rm tot}}{dt}= \llbracket f\rrbracket^*_{r} - \llbracket f\rrbracket^*_{-r} + \tau\kappa\mathcal{D}\overline\zeta \frac{d\overline\zeta}{dt}+ \eps\mu \int_{-\infty}^{\ell}R
	\end{equation*} and, invoking item \eqref{tot-ene-conservation} of Assumption \ref{assumption}, we find \eqref{f-con}. Recalling the definition of the energy flux \eqref{flux} and combining item \eqref{discharge-con} of Assumption \ref{assumption} with \eqref{NSW-interior}, we obtain \eqref{bra-con}. Furthermore, using \eqref{equiv-brackets}, we obtain \eqref{bra-con-equi} after observing that \eqref{NSW-interior} implies $\llbracket \zeta_{|_{\mathcal{I}}} + \mathfrak{E}_{|_{\mathcal{I}}}\rrbracket=0$.
\end{proof} 
An immediate consequence of Proposition \ref{prop-conserve-coup} is the following conditional result for energy conservation at order $O(\eps\mu)$, involving shallow water analogues of Bernoulli's principle.
\begin{proposition} \label{prop-intpres-con}
Considering the shallow water Bernoulli principles at the contact points
	\begin{align}
		\label{bernoulli-left}	&\zeta_{\rm w} +\frac{\underline{P}_{|_{\mathcal{I}}}}{\eps}+ \eps\frac{q_i^2}{2h_{\rm w}^2}= \zeta_{|_{\mathcal{E}}} 
		+\mathfrak{E}_{|_{\mathcal{E}}}  
		&\quad \text{at} \quad x=-r, \\[5pt]\label{bernoulli-right}	
		&\zeta_{\rm w} +\frac{\underline{P}_{|_{\mathcal{I}}}}{\eps}+ \eps\frac{q_i^2}{2h_{\rm w}^2}= \zeta_{|_{\mathcal{E}}}  +P_0(I - \kappa \mathcal{D})\overline{\zeta} 
		+\mathfrak{E}_{|_{\mathcal{E}}} 
		&\quad \text{at} \quad x= r,
	\end{align}with $\mathfrak{E}$ as in \eqref{frak-E},
	regular solutions to the constrained shallow water models \eqref{dimless-shallow-set}-\eqref{contact-con} and \eqref{pressure-mean} with $\kappa=0$, conserve the total energy \eqref{total-energy} at order $O(\eps\mu)$. 
\end{proposition}
\begin{proof}Using \eqref{contact-con}-\eqref{NSW-interior} and \eqref{pressure-mean}, the jumps at the contact points in \eqref{bra-con} respectively read
	\begin{equation*}
		\left\llbracket \zeta+ \frac{\underline{P}}{\eps} +\mathfrak{E} \right\rrbracket^*_{-r}=	\zeta_{|_{\mathcal{E}}}(-r) + \mathfrak{E}_{|_{\mathcal{E}}} (-r) - \zeta_{\rm w}(-r)-\eps \frac{q_i^2}{2h_{\rm w}^2(-r)}   - \frac{\underline{P}_{|_{\mathcal{I}}}(-r)}{\eps}
	\end{equation*}and 
	\begin{equation*}
		\left \llbracket \zeta+\frac{\underline{P}}{\eps} +\mathfrak{E} \right \rrbracket^*_{r} =	\zeta_{|_{\mathcal{E}}}(r)+ P_0(I - \kappa \mathcal{D})\overline{\zeta} + \mathfrak{E}_{|_{\mathcal{E}}} (r)- \zeta_{\rm w}(r)  - \eps \frac{q_i^2}{2h_{\rm w}^2(r)}  - \frac{\underline{P}_{|_{\mathcal{I}}}(r)}{\eps},
	\end{equation*}
	where we have used that \eqref{frak-E} in the interior domain reduces to \begin{equation} \label{frakE-red} \mathfrak{E}_{|_{\mathcal{I}}}= \eps\frac{q_i^2}{2h_{\rm w}^2}\end{equation}since the dispersive term vanishes as $\partial_x\partial_t q_i = 0$. Therefore, if \eqref{bernoulli-left}-\eqref{bernoulli-right} are satisfied, then 
	\begin{equation*}
		\left\llbracket \zeta+ \frac{\underline{P}}{\eps}  +\mathfrak{E}\right\rrbracket^*_{-r}= 0=	\left\llbracket \zeta + \frac{\underline{P}}{\eps}+\mathfrak{E} \right\rrbracket^*_{r}
	\end{equation*}and 
	\eqref{bra-con} holds. Thanks to Proposition \ref{prop-conserve-coup}, we conclude that the total energy is conserved at order $O(\eps\mu)$.
\end{proof}
\begin{remark}
	Note that in the symmetric case $\zeta_{\rm w}(-r)=\zeta_{\rm w}(r)$, which holds for a solid with flat bottom, in order to satisfy \eqref{bra-con} one can also consider the coupling conditions
	\begin{equation}\begin{aligned}\label{pre-boundary-sym}
			\frac{\underline{P}_{|_{\mathcal{I}}}}{\eps}= \zeta_{|_{\mathcal{E}}} 
			+\mathfrak{E}_{|_{\mathcal{E}}}  
			\quad \text{at} \quad x=-r, \qquad 
			\frac{\underline{P}_{|_{\mathcal{I}}}}{\eps}= P_0(I - \kappa \mathcal{D})\overline{\zeta} +\zeta_{|_{\mathcal{E}}} 
			+\mathfrak{E}_{|_{\mathcal{E}}}  
			\quad \text{at} \quad x= r.
		\end{aligned}
	\end{equation}
	However, we choose to consider \eqref{bernoulli-left}-\eqref{bernoulli-right} since they hold in the general non-flat case and, most importantly, are expressed in the form of Bernoulli's principles, with the sum of the static pressure $\zeta + \frac{\underline{P}}{\eps}$ and the shallow water dynamic pressure $\mathfrak{E}$ appearing explicitly.
\end{remark}

Shallow water Bernoulli principles \eqref{bernoulli-left}-\eqref{bernoulli-right} provide boundary conditions for the surface pressure in the interior domain at the contact points analogous to those derived in \cite[Corollary 2.1]{BecLan22} for a vertically moving solid. They take the form of the sum of the surface pressure deviation in the exterior domain and the jumps of the surface elevation and the shallow water dynamic pressure between the exterior and interior domains. In \cite{BecLan22}, the dispersive term in the dynamic pressure $\mathfrak{E}_{|_{\mathcal{I}}}$ does not vanish as it depends on the motion of the solid and the surface pressure deviation vanishes in the whole exterior domain. Here, instead, there is no dispersive term in \eqref{frakE-red} and $\underline{P}$ does not vanish in the chamber domain since it depends on the mean surface elevation $\overline{\zeta}$.\\
Differentiating the momentum equation in \eqref{NSW-interior} with respect to $x$ and using \eqref{bernoulli-left} and \eqref{bernoulli-right} allows us to fully determine the surface pressure in the interior domain. More precisely, it is the unique solution to the Dirichlet problem
\begin{equation*}
	\begin{cases}
		\partial_x^2\left(	\dfrac{	\underline{P}}{\eps}\right)=0  &\qquad\text{in} \quad \mathcal{I}\\[7pt]
		\dfrac{	\underline{P}}{\eps}= \zeta_{|_{\mathcal{E}}} -\zeta_{\rm w} + \mathfrak{E}_{|_{\mathcal{E}}}- \eps \dfrac{q_i^2}{2h_{\rm w}^2} \quad &\qquad\text{at} \quad x=-r\\[7pt]	\dfrac{	\underline{P}}{\eps}= P_0(I-\kappa \mathcal{D})\overline\zeta + \zeta_{|_{\mathcal{E}}}-\zeta_{\rm w} + \mathfrak{E}_{|_{\mathcal{E}}} - \eps \dfrac{q_i^2}{2h_{\rm w}^2}\quad &\qquad\text{at} \quad x=r.
	\end{cases}
\end{equation*}

\subsection{Transmission problems between the open water and the chamber}\label{subsec-transmission}
In the previous subsection we have reformulated the shallow water models in the presence of an OWC as the coupled problems \eqref{NSW-exterior}-\eqref{NSW-interior}  with coupling conditions given by item \eqref{discharge-con} of Assumption \ref{assumption} and \eqref{bra-con}. However, the fluid dynamics can also be formulated as transmission problems between $\mathcal{E}_-$ and $\mathcal{E}_+$ across the the contact points.
After combining \eqref{bra-con-equi} with \eqref{pressure-mean}, we have that 
\begin{equation}\label{jump-Pi}
	\frac{\llbracket\underline{P}_{|_{\mathcal{I}}}\rrbracket}{\eps} = \llbracket \zeta_{|_{\mathcal{E}}} + \mathfrak{E}_{|_{\mathcal{E}}} \rrbracket + P_0\left( I-\kappa \mathcal{D}\right)\overline{\zeta}.
\end{equation}
Integrating the momentum equation in \eqref{NSW-interior} over $\mathcal{I}$ and using \eqref{jump-Pi} then yields
\begin{equation}\label{evo-qi}
	\alpha	\frac{d q_i}{dt}=-  \llbracket \zeta_{|_{\mathcal{E}}} + \mathfrak{E}_{|_{\mathcal{E}}} \rrbracket - P_0(I -\kappa \mathcal{D})\overline{\zeta}
\end{equation}with $\alpha= 2r/h_{\rm w}$, relating the evolution of the interior discharge to the jump of the fluid unknowns in the exterior domain across the contact points and to the mean surface elevation in the chamber. Analogous evolution equations were derived in \cite{BecLan22, BocHeVer23}. Here, differently from \cite{BocHeVer23}, the shallow water dynamic pressure $\mathfrak{E}$ includes the dispersive contribution when $\mu\neq0$, whereas, differently from \cite{BecLan22}, the mean surface term appears due to the air pressure variation in the OWC chamber.
Combining item \eqref{discharge-con} of Assumption \ref{assumption} with \eqref{NSW-interior} then yields the transmission conditions for the discharge in the exterior domain
\begin{equation}\label{trans-con}
	\llbracket q_{|_{\mathcal{E}}} \rrbracket=0, \qquad \langle q_{|_{\mathcal{E}}}\rangle = q_i.
\end{equation}
Gathering \eqref{NSW-exterior}, \eqref{evo-qi} and \eqref{trans-con} together leads to the following transmission problem between the open water and the chamber domains across the contact points. 

Before stating it, we introduce the notation $\widetilde{\mu}=\sqrt{\dfrac{\mu}{3}}$ and, from now on, we use it instead of $\mu$ when writing dispersive terms.

\begin{proposition}\label{prop-transmission}
	Under Assumption \ref{assumption}, the constrained shallow water models \eqref{dimless-shallow-set}-\eqref{contact-con} and \eqref{pressure-mean} can be recast as the transmission problems
	\begin{equation}\label{trans-sys}
		\begin{cases}
			\partial_t \zeta + \partial_x q=0 \\[5pt]
			\Big(1-\widetilde{\mu}^2\partial_x^2\Big) \partial_t q + \eps \partial_x \Big(\dfrac{q^2}{h}\Big) + h \partial_x \zeta=0
		\end{cases}\quad \text{in}\quad \mathcal{E}
	\end{equation} with boundary conditions \eqref{bound-cond} and transmission conditions
	\begin{equation}\label{trans-conditions}
		\llbracket q \rrbracket=0, \qquad \langle q\rangle = q_i,
	\end{equation}where $\llbracket \cdot \rrbracket$ and $\langle \cdot \rangle$ defined in \eqref{jump-mean}, and $q_i$ satisfies the evolution equation \begin{equation}\label{evo-eq-qi}
		\alpha\frac{d q_i}{dt}=-  \llbracket \zeta + \mathfrak{E} \rrbracket - P_0(I-\kappa \mathcal{D})\overline{\zeta}
	\end{equation}
	with $\alpha=2r/h_{\rm w}$, $\mathfrak{E}$ as in \eqref{frak-E} and $\mathcal{D}$ as in \eqref{delay}.
\end{proposition}
In fact, we know from \cite{BecLan22} that in the dispersive case ($\widetilde{\mu}\neq0$) the jump in \eqref{evo-eq-qi} contains a leading-order term and hence it must be moved to the left hand-side of \eqref{evo-eq-qi}. This fact has important consequences for simulations, since it avoids numerical instabilities  when treating the additional leading-order term as part of the left-hand side of the evolution equation. In order to show its explicit expression, let us introduce the momentum flux 
\begin{equation}\label{momentum-flux}
	\varphi(\zeta, q)=\zeta + \eps\frac{\zeta^2}{2} +\eps \frac{q^2}{h}
\end{equation} and the operators $R_0, R_1$ defined as the inverses of  $I- \widetilde{\mu}^2 \partial_x^2$ complemented respectively with homogeneous Dirichlet and Neumann boundary conditions, that is, 
\begin{equation}\label{R0}
	R_0 v= u \iff \begin{cases}
		(I- \widetilde{\mu}^2 \partial_x^2) u= v\qquad  \text{in} \quad \mathcal{E}\\
		u\rightarrow 0 \ \text{as}\ x\rightarrow -\infty, \quad u=0 \ \text{at} \ x={\pm r}, \ell
	\end{cases}
\end{equation}
and 
\begin{equation}\label{R1}
	R_1 v= u \iff \begin{cases}
		(I- \widetilde{\mu}^2 \partial_x^2) u= v\qquad  \text{in} \quad \mathcal{E}\\
		\partial_xu\rightarrow 0 \ \text{as}\ x\rightarrow -\infty, \quad \partial_xu=0 \ \text{at} \ x={\pm r}, \ell
	\end{cases}
\end{equation}
\begin{proposition}\label{prop-evo-qi}
	The evolution equation \eqref{evo-eq-qi} of the transmission problem in Proposition \ref{prop-transmission} can be reformulated as the following nonlinear first-order ODEs:
	for $\widetilde{\mu}=0$,
	\begin{equation}\label{added-evo-qi-0}
		\alpha \frac{dq_i}{dt} + \frac{ \eps}{2}\left \llbracket\frac{1}{h^2}\right\rrbracket  q_i^2=- \left\llbracket \zeta\right\rrbracket  - P_0 \left(I-\kappa\mathcal{D}\right)\overline\zeta,
	\end{equation}
	and for $\widetilde{\mu}\neq0$
	\begin{equation}\label{added-evo-qi}
		\left(\alpha + \mathcal{T}_{\widetilde{\mu}}(h)\right)\frac{dq_i}{dt} - \frac{ \eps}{2}\left \llbracket\frac{1}{h^2}\right\rrbracket  q_i^2=- \left\llbracket \mathcal{F} \right\rrbracket  - P_0 \left(I-\kappa\mathcal{D}\right)\overline\zeta,
	\end{equation} where 
	\begin{equation}\label{T-mu}
		\mathcal{T}_{\widetilde{\mu}}(h)=\widetilde{\mu}  \left(C_{1,\widetilde{\mu}} \left\langle \frac{1}{h}\right\rangle + C_{2,\widetilde{\mu}} \left\llbracket \frac{1}{h}\right\rrbracket\right)>0,
	\end{equation}
	with $$C_{1,\widetilde{\mu}}= \frac{ e^{\frac{\ell-r}{\widetilde{\mu}}}}{\sinh(\frac{\ell-r}{\widetilde{\mu}})}, \qquad C_{2,\widetilde{\mu}}= \frac{e^{-\frac{\ell-r}{\widetilde{\mu}}}}{ 2\sinh(\frac{\ell-r}{\widetilde{\mu}})},$$
	and 
	\begin{equation}\label{calF}
		\mathcal{F}= \dfrac{1}{h} \left( \eps \frac{\zeta^2}{2} +R_1 \varphi \right)
	\end{equation}with 
	$\varphi, R_1$ as in \eqref{momentum-flux} and \eqref{R1}. 
\end{proposition}
\begin{proof}
	When  $\widetilde{\mu}=0$ (equivalently, $\mu=0$ in \eqref{frak-E}), the shallow water dynamic pressure reads 
	$\mathfrak{E}  = 	 \dfrac{\eps}{2} \dfrac{q^2}{h^2}.$
	Using the identity $\llbracket fg\rrbracket=\llbracket f\rrbracket \langle g \rangle+  \langle f\rangle \llbracket g\rrbracket$ and since \eqref{trans-conditions} implies that $\llbracket q^2\rrbracket=0$ and $\langle q^2\rangle=q_i^2,$ we have that 
	\begin{equation}\label{jump-qi}
		\frac{\eps}{2}\left\llbracket \frac{q^2}{h^2}\right\rrbracket=  \frac{\eps}{2}\left\llbracket \frac{1}{h^2}\right\rrbracket q_i^2
	\end{equation}
	and \eqref{added-evo-qi-0}
	follows directly from
	\eqref{evo-eq-qi}. 	
	Let us now consider $\widetilde{\mu}\neq0$.
	Thanks to the relation $f(\pm r)= \langle f\rangle \pm \llbracket f\rrbracket/2$, differentiating \eqref{bound-cond} and \eqref{trans-conditions} with respect to time implies that $$\partial_t q \rightarrow 0 \quad \text{as} \quad x\rightarrow -\infty, \qquad  	\partial_t q (\pm r)= \dfrac{dq_i}{dt}, \qquad \partial_t q (\ell)=0.$$ Therefore, using \eqref{momentum-flux}-\eqref{R0}, we can write $\partial_t q$ in the momentum equation of \eqref{trans-sys} explicitly as
	\begin{equation}\label{explicit-timederq}\begin{aligned}
		 	 \partial_t q = -R_0\partial_x \varphi + \frac{dq_i}{dt} e^{\frac{x+r}{\widetilde{\mu}}}\quad \text{in} \quad \mathcal{E}_-,\quad \
				\partial_t q = -R_0\partial_x \varphi + \frac{dq_i}{dt} \frac{\sinh(\frac{\ell-x}{\widetilde{\mu}})}{\sinh(\frac{\ell-r}{\widetilde{\mu}})} 
			\quad \text{in} \quad \mathcal{E}_+
		\end{aligned}
	\end{equation}with $\varphi$ defined in \eqref{momentum-flux}, where the second terms on the right-hand sides of the equalities in \eqref{explicit-timederq} are respectively the solutions to 
	\begin{equation*}
		\begin{cases}
			\left(I-\widetilde{\mu}^2 \partial_x^2\right)u=0 \qquad \text{in} \quad \mathcal{E}_-\\
			u \rightarrow 0 \quad \text{as} \quad x\rightarrow -\infty, \qquad u= \dfrac{dq_i}{dt} \quad\text{at}\quad x=-r
		\end{cases}
	\end{equation*}and
	\begin{equation*}
		\begin{cases}
			\left(I-\widetilde{\mu}^2 \partial_x^2\right)u=0 \qquad \text{in} \quad \mathcal{E}_+\\
			u= \dfrac{dq_i}{dt} \quad\text{at}\quad x=r, \qquad u =0 \quad \text{at} \quad x=\ell.
		\end{cases}
	\end{equation*}
	Differentiating \eqref{explicit-timederq} with respect to $x$ and dividing by $h$ permits us to rewrite the dispersive jump term in \eqref{evo-eq-qi} as
	\begin{equation}\label{disp-jump}
		\begin{aligned}
			-\widetilde{\mu}^2\left\llbracket \frac{\partial_x\partial_t q }{h}\right\rrbracket&= \widetilde{\mu}^2\left\llbracket \frac{\partial_xR_0\partial_x \varphi}{h}\right\rrbracket +\widetilde{\mu} \frac{dq_i}{dt} \Big(  
			\frac{\cosh(\frac{\ell-r}{\widetilde{\mu}})}{\sinh(\frac{\ell-r}{\widetilde{\mu}})}
			\frac{1}{h(r)} +  \frac{1}{h(-r)} \Big)\\
		&=\widetilde{\mu}^2\left\llbracket \frac{\partial_xR_0\partial_x \varphi}{h}\right\rrbracket +\frac{dq_i}{dt} \frac{\widetilde{\mu}}{ 
			\sinh(\frac{\ell-r}{\widetilde{\mu}})
		}\Big(   e^{\frac{\ell-r}{\widetilde{\mu}}} \left\langle \frac{1}{h}\right\rangle  + \frac{e^{-\frac{\ell-r}{\widetilde{\mu}}}}{2} \left\llbracket \frac{1}{h} \right\rrbracket\Big). 
	\end{aligned}
\end{equation}
Thanks to the identity $R_0\partial_x = \partial_x R_1$ and using the definition of $R_1$, we write
\begin{equation}\label{R0-R1}\widetilde{\mu}^2 \partial_x R_0 \partial_x \varphi = 	\widetilde{\mu}^2 \partial_x^2 R_1 \varphi= (R_1-I) \varphi,
\end{equation}which, together with \eqref{momentum-flux}, implies that 
\begin{equation}\label{jump-term}
	\zeta + \eps\frac{q^2}{2h^2} +\widetilde{\mu}^2\frac{\partial_x R_0 \partial_x \varphi}{h}= 	\zeta + \eps\frac{q^2}{2h^2}+ \frac{(R_1 - I)\varphi}{h}=	\frac{1}{h}\Big( \eps\frac{\zeta^2}{2}+ R_1\varphi\Big) -\eps\frac{q^2}{2h^2}.
\end{equation}
Invoking \eqref{jump-qi} and introducing the notations \eqref{T-mu}-\eqref{calF}, we finally combine \eqref{disp-jump} and \eqref{jump-term} to write the jump term in \eqref{evo-eq-qi} as 
\begin{equation*}
	\left \llbracket \zeta + \mathfrak{E} \right\rrbracket=\left \llbracket  \mathcal{F}    \right\rrbracket + \mathcal{T}_{\widetilde{\mu}}(h)\dfrac{dq_i}{dt} -\frac{\eps}{2}\left\llbracket \frac{1}{h^2}\right \rrbracket q_i^2, 
\end{equation*}which yields \eqref{added-evo-qi}.
\end{proof}
\begin{remark} The positivity of the dispersive contribution $\mathcal{T}_{\widetilde{\mu}}(h)$ in \eqref{T-mu} guarantees the solvability of \eqref{added-evo-qi}. In the limiting case of an OWC chamber with infinite width, the term due to the jump disappears and we recover the same dispersive contribution as in the absence of the chamber \cite{BecLan22}, namely,
\begin{equation*}
	\mathcal{T}_{\widetilde{\mu}}(h)  \rightarrow  2\widetilde{\mu} \left\langle \frac{1}{h} \right\rangle \quad \text{as}\quad |\mathcal{E}_+|=\ell-r\rightarrow+\infty.
\end{equation*}
\end{remark}

\subsection{Well-posedness of the initial boundary value problems}
In this subsection, we discuss the well-posedness of the initial boundary value problems associated with the transmission problems \eqref{trans-sys}-\eqref{trans-conditions} and \eqref{added-evo-qi-0} for $\widetilde{\mu}=0$, and \eqref{trans-sys}-\eqref{trans-conditions} and \eqref{added-evo-qi} for $\widetilde{\mu}\neq0$.
To this end, we complement the transmission problems with the initial conditions
\begin{equation}\label{ini-conds-trans}
	(\zeta, q)(0, \cdot)= (\zeta^{\rm in}, q^{\rm in}) \quad \text{in} \quad \mathcal{E}, \qquad  q_i(0)= q_i^{\rm in}.
\end{equation}

\smallskip

\emph{Hyperbolic case.} To study the well-posedness of the initial boundary value problem in the case $\widetilde{\mu}=0$, we consider a slightly different problem. Rather then setting the nonlinear shallow water equations in the unbounded domain $(-\infty, \ell)$, we introduce a non-physical left boundary at $x=-\ell$, where we prescribe a \emph{generating} boundary condition. More precisely, we assume that the fluid surface elevation is known
(typically measured by floating buoys) and equal to a given entry time-dependent function $f_{\rm ent}$.
The resulting transmission problem then reads 
\begin{equation}\label{NSW-sym}
	\begin{cases}
		\partial_t \zeta + \partial_x q=0 \\[5pt]
		\partial_t q + \eps \partial_x \Big(\dfrac{q^2}{h}\Big) + h \partial_x \zeta=0
	\end{cases}\quad \text{in}\quad \mathcal{E}_{\rm sym}=(-\ell, -r)\cup (r, \ell)
\end{equation} with boundary conditions
\begin{equation*}
	\zeta = f_{\rm ent} \quad \text{at}\quad {x=-\ell}, \qquad q=0 \ \quad \text{at}\quad {x=\ell},
\end{equation*} 
and transmission conditions
\begin{equation}\label{tra-con}
	\llbracket q \rrbracket=0, \qquad \langle q\rangle = q_i,
\end{equation}where $q_i$ satisfies
\begin{equation}\label{evo-qi-NSW}
	\alpha \frac{dq_i}{dt} + \frac{ \eps}{2}\left \llbracket\frac{1}{h^2}\right\rrbracket  q_i^2=- \left\llbracket \zeta\right\rrbracket  - P_0 \left(I-\kappa\mathcal{D}\right)\overline\zeta.
\end{equation} 

The reason for working with the bounded domain $\mathcal{E}_{\rm sym}$ instead of the unbounded exterior domain $\mathcal{E}$ is purely technical. Indeed, the well-posedness analysis relies on the reformulation of the $2\times2$ quasilinear system set in $\mathcal{E}$ into a $4\times4$ quasilinear system set in a single connected component of $\mathcal{E}$, say $\mathcal{E}_+$. To preserve the autonomous quasilinear structure of the system, we must perform an affine change of variables, which cannot map unbounded intervals onto bounded ones.\\
Before stating the local well-posedness result, we introduce the functional space to which the solution will belong. Let $m\geq0$ be an integer, $\Omega\subseteq \mathbb{R}$ and $T>0$. We define 
\begin{equation*}
	\mathbb{W}^m_T(\Omega)= \bigcap_{j=0}^m C^j([0,T]; H^{m-j}(\Omega))
\end{equation*}endowed with the norm
$$\|(\zeta, q)\|_{\mathbb{W}^m_T(\Omega)}=\sup_{t\in [0,T]}\sum_{j=0}^{m}\|(\partial_t^j \zeta, \partial_t^j q)(t, \cdot)\|_{H^{m-j}(\Omega)}. $$
\begin{theorem}\label{theo-trans-wp}
	Let $m\geq 2$ be an integer. Consider $(\zeta^{\rm in}, q^{\rm in}) \in H^m (\mathcal{E}_{\rm sym})$ such that 
	\begin{equation}\label{pos-assumption}
		\inf_{ \mathcal{E}}\Big(\sqrt{1+\eps \zeta^{\rm in}} - \eps\Big|\frac{q^{\rm in}}{1+\eps \zeta^{\rm in}}\Big| \Big) >0,
	\end{equation}  $q_i^{\rm in}\in \mathbb{R}$ and $f_{\rm ent}\in H^m(0,T)$ for some $T>0$, verifying
	suitable compatibility conditions. Then, for any $\eps\in (0,1]$, there exist $T_1\in (0, T]$ and a unique solution $(\zeta, q, q_i)$ to system \eqref{NSW-sym}-\eqref{evo-qi-NSW}, with $(\zeta, q)\in \mathbb{W}^m_{T_1}(\mathcal{E}_{\rm sym})$ and $q_i\in H^{m+1}(0,T_1)$, satisfying the initial conditions \begin{equation*}
		(\zeta, q)(0, \cdot)= (\zeta^{\rm in}, q^{\rm in}) \quad \text{in} \quad \mathcal{E}_{\rm sym}, \qquad  q_i(0)= q_i^{\rm in}.
	\end{equation*}
\end{theorem}

\begin{proof}This result is essentially \cite[Theorem 4.15]{BocHeVer23}, which is formulated for a slightly different version of the transmission problem (with dimensional equations). Before outlining the main steps of its proof, we first rewrite the problem in the same formulation by splitting \eqref{evo-qi-NSW} into a system of two coupled ODEs with unknowns $(q_i, P_{\rm ch})$. Indeed, after noticing that item \eqref{discharge-con} of Assumption \ref{assumption} together with the wall condition at $x=\ell$ yields that
	\begin{equation}\label{mean-qi}
		\dfrac{d\overline\zeta}{dt} = \frac{q_i}{|\mathcal{E}_+|},
	\end{equation}we use \eqref{evo-Pch-dimless} and \eqref{Pch-dimless}
	to write \eqref{evo-qi-NSW} as the system 
	\begin{equation}\label{ODE-system}
		\begin{cases}
			\alpha \dfrac{dq_i}{dt}+ \dfrac{\eps}{2} \left\llbracket\dfrac{1}{h^2}\right\rrbracket q_i^2= - \left\llbracket \zeta \right\rrbracket - \dfrac{P_0}{\epsilon}P_{\rm ch}\\[7pt]
			\dfrac{dP_{\rm ch}}{dt}  + \kappa P_{\rm ch} =\dfrac{\epsilon}{|\mathcal{E}_+|}q_i, \quad P_{\rm ch}(0)= \dfrac{\epsilon}{|\mathcal{E}_+|}\overline\zeta(0).
		\end{cases}
	\end{equation}
	Now we are in the same setting as in \cite[Theorem 4.15]{BocHeVer23}. In view of the change of variable $x\mapsto -x$ in $(-\ell, -r)$, after denoting $(\zeta^-,q^-)(\cdot, x)= (\zeta,q)(\cdot, -x)$ and introducing the unknowns $u=(\eta, q, \eta^- , q^-)$ and $G=(q_i, P_{\rm ch})$, the transmission problem \eqref{NSW-sym}-\eqref{tra-con} and \eqref{ODE-system} can be reformulated as a $4\times 4$ quasilinear hyperbolic system set in the chamber domain in the form
	\begin{equation}\label{4x4quasihyp}
		\begin{cases}
			\partial_t u + \mathcal{A}(u)\partial_x u = 0 \quad &\text{in}\quad  \mathcal{E}_+\\
			\mathcal{M}_r u=V(G(t)) \quad &\text{at}\quad x=r\\
			\mathcal{M}_\ell u=g(t)\quad &\text{at}\quad x=\ell
		\end{cases}
	\end{equation} with a semilinear boundary condition determined by a first-order ODE 
	\begin{equation}\label{semilinearODE}
		\frac{dG}{dt}= \Theta\left(G, u(r) \right)
	\end{equation} containing terms that involve the trace of $u$ at $x=r$. The main ingredient in the analysis of the initial boundary value problem associated with \eqref{4x4quasihyp}-\eqref{semilinearODE} is the construction of a Kreiss symmetrizer that makes the boundary conditions maximally dissipative. This property allows one to show a hidden trace regularity of the PDE solution in the a priori estimates for the linearized version of \eqref{4x4quasihyp}. The explicit construction of such a symmetrizer requires the assumption of the subsonic regime condition \eqref{pos-assumption}, propagated by the dynamics, that guarantees the existence of two incoming characteristics and ensures that an analogue of the uniform Kreiss-Lopatinskii condition is satisfied.
	Moreover, seeking solutions enjoying the classical hyperbolic (integer) Sobolev regularity $m\geq 2$ requires suitable compatibility conditions (see \cite[Definition 4.12]{BocHeVer23} for the precise statement). Taking advantage of the extra trace regularity control, it is possible to show the convergence of an iterative scheme on the coupled PDE-ODE problem \eqref{4x4quasihyp}-\eqref{semilinearODE} and to find the solution as its limit.\\
	Applying \cite[Theorem 4.15]{BocHeVer23} then yields local existence and uniqueness of the solution $(\zeta, q, q_i, P_{\rm ch})$ to \eqref{NSW-sym}-\eqref{tra-con} and \eqref{ODE-system}, with $(\zeta,q)\in \mathbb{W}^m_{T_1}(\mathcal{E}_{\rm sym})$ verifying the extra trace regularity
	\begin{equation*}
		\sum_{j=0}^{m}	\Big(\|(\partial_x^j\zeta, \partial_x^jq)(\cdot,\pm r)\|_{H^{m-j}(0,T_1)} + \|(\partial_x^j\zeta, \partial_x^jq)(\cdot, \pm \ell)\|_{H^{m-j}(0,T_1)}	\Big) < \infty,
	\end{equation*}
	$(q_i, P_{\rm ch})\in H^{m+1}(0,T_1)$ for $m\geq 2$, satisfying the initial conditions
	\begin{equation}\label{conds-in}
		(\zeta, q)(0, \cdot)= (\zeta^{\rm in}, q^{\rm in}) \quad \text{in} \quad \mathcal{E}_{\rm sym}, \qquad  (q_i, P_{\rm ch})(0)= \Big(q_i^{\rm in},P_{\rm ch}^{\rm in}\Big),
	\end{equation} with $P_{\rm ch}^{\rm in}={\epsilon}\overline{\zeta^{\rm in}}$.
	In particular, using the second equation in \eqref{ODE-system}, we have that $P_{\rm ch}$ additionally belongs to $H^{m+2}(0,T_1)$. Coming back from \eqref{ODE-system} to \eqref{evo-qi-NSW} then concludes the proof.
	
\end{proof}

\smallskip

\emph{Dispersive case.} In the case $\widetilde{\mu}\neq 0$, we consider the transmission problem \eqref{trans-sys}-\eqref{evo-eq-qi} and \eqref{added-evo-qi} set in the original exterior domain $\mathcal{E}$. The following local well-posedness result holds:

\begin{theorem}\label{theo-trans-wp-disp}
	Let $m\geq 1$ be an integer. Consider $(\zeta^{\rm in}, q^{\rm in}) \in H^m(\mathcal{E})\times H^{m+1}(\mathcal{E}) $,  $q_i^{\rm in}\in \mathbb{R}$ such that
	\begin{equation}\label{initial-ass}
		\inf_{ \mathcal{E}}(1+\eps \zeta^{\rm in}) >0, \quad \llbracket q^{\rm in}\rrbracket =0, \quad \langle q^{\rm in}\rangle=q_i^{\rm in} \quad \text{and}\quad q^{\rm in} =0 \ \text{ at } \ x=\ell.
	\end{equation} 
	Then, for any $\eps\!\in \!(0,1]$ and $\widetilde{\mu} \!>\!0$, there exists $T\!>\!0$  such that the transmission problem
	\eqref{trans-sys}-\eqref{evo-eq-qi} and \eqref{added-evo-qi} admits a unique solution $(\zeta, q, q_i)\!\in \!C^1\!\left([0,T); \!H^m(\mathcal{E})\!\times\! H^{m+1}\!(\mathcal{E})\!\times\! \mathbb{R}\!\right)$ satisfying the initial conditions \eqref{ini-conds-trans}.
	In addition, denoting by $T_{\rm max}$ the maximal existence time, if $T_{\rm max}$ is finite then 
	\begin{equation}
		\limsup_{t\rightarrow T_{\rm max}^-}\Big( \Big\| \zeta(t), q(t), \frac{1}{h(t)}\Big\|_{L^\infty(\mathcal{E})}  + |q_i(t)|\Big) = +\infty.
	\end{equation} 
	
\end{theorem}
\begin{remark}\label{rem-uniform}
	We stress that the existence time $T>0$ of Theorem \ref{theo-trans-wp-disp} may shrink as $\widetilde{\mu}$ goes to zero. In order to prove that it is of order $O(1/(\eps + \widetilde{\mu}^2))=O(1/\eps)$ uniformly with respect to $\widetilde{\mu}$, as in the case of Boussinesq systems on the full line, one has to derive uniform estimates. This was done in \cite{BreLanMet21} for a formally equivalent Boussinesq system, where the dispersive term in \eqref{evo-eq-qi} is linear. To this end, compatibility conditions in the form of inequalities (instead of equalities for the hyperbolic case) must be required. In our case, the nonlinearity of the dispersive term makes the analysis much more delicate and following \cite{BecLan22} the same time-scale $O(1/\eps)$ can be obtained as a conditional result assuming $(\zeta, q)$ are uniformly bounded in $W^{1,\infty}(\mathcal{E})$. The authors were able to derive uniform estimates at time-scale $O(1/\sqrt{\eps})$ for the linearized version of \eqref{trans-sys} showing some hidden trace regularity pointwise in time due to dispersion. As they claimed, the question of whether this shorter time-scale is dictated by the dispersive control of the traces or it is only a technical limitation is open. We refer to \cite[Section 3]{BecLan22} for a deeper discussion.
\end{remark}

\begin{proof}
	The proof follows the steps of \cite[Theorem 3.3]{BecLan22} and relies on the fact that the problem in the dispersive case can be reduced to an infinite-dimensional ODE. First, analogously to the hyperbolic case, we rewrite  \eqref{added-evo-qi} as the coupled system 
	\begin{equation}\label{ODE-system-disp}
		\begin{cases}
			\left(\alpha + \mathcal{T}_{\widetilde{\mu}}(h)\right)\dfrac{dq_i}{dt} - \dfrac{ \eps}{2}\left \llbracket\dfrac{1}{h^2}\right\rrbracket  q_i^2=- \left\llbracket \mathcal{F} \right\rrbracket  - \dfrac{P_0}{\epsilon}P_{\rm ch},\\[7pt]
			\dfrac{dP_{\rm ch}}{dt}  + \kappa P_{\rm ch} =\dfrac{\epsilon}{|\mathcal{E}_+|}q_i, \quad P_{\rm ch}(0)= \dfrac{\epsilon}{|\mathcal{E}_+|}\overline\zeta(0).
		\end{cases}
	\end{equation}
	Using the reformulations of the momentum equations \eqref{explicit-timederq} and the positivity of $\alpha + \mathcal{T}_{\widetilde{\mu}}(h)$ whenever  $h$ does not vanish, we write the system  \eqref{trans-sys}-\eqref{trans-conditions} and \eqref{ODE-system-disp} as 
	\begin{equation}\label{ODE-0}
		\begin{cases}	
			\partial_t \zeta = -\partial_x q \qquad &\text{in} \quad \mathcal{E}\\[5pt]
			\partial_t q = -R_0\partial_x \varphi + \dfrac{dq_i}{dt} \Big( e^{\frac{x+r}{\widetilde{\mu}}} \chi_- + \frac{\sinh(\frac{\ell-x}{\widetilde{\mu}})}{\sinh(\frac{\ell-r}{\widetilde{\mu}})}\chi_+\Big)\qquad &\text{in} \quad \mathcal{E}\\[7pt]
			\dfrac{dq_i}{dt} = \dfrac{1}{	\alpha + \mathcal{T}_{\widetilde{\mu}}(h)}  \Big( \dfrac{ \eps}{2}\left \llbracket\dfrac{1}{h^2}\right\rrbracket  q_i^2- \left\llbracket \mathcal{F} \right\rrbracket  - \dfrac{P_0}{\epsilon}P_{\rm ch}\Big)\\[7pt]
			\dfrac{dP_{\rm ch}}{dt}  =- \kappa P_{\rm ch} +\dfrac{\epsilon}{|\mathcal{E}_+|}q_i
		\end{cases}
	\end{equation}
	where $\chi_\pm$ are the characteristic functions of $\mathcal{E}_\pm$. We now denote by $U= (U_1, U_2, U_3, U_4)$ and  $U^{\rm in}$ respectively the quadruples of unknowns $(\zeta, q, q_i, P_{\rm ch})$ and initial data $(\zeta^{\rm in}, q^{\rm in}, q_i^{\rm in}, P_{\rm ch}^{\rm in})$, with $P_{\rm ch}^{\rm in}$ as in \eqref{conds-in}. Then, after plugging the third line of \eqref{ODE-0} into the second one, we write the Cauchy problem related to \eqref{ODE-0} in the compact form
	\begin{equation}\label{compact-ODE}\begin{cases}
			\dfrac{d U}{dt}= \Phi (U),\\[5pt]
			U(0)= U^{\rm in},
		\end{cases}
	\end{equation}with 
	\begin{align*}
		&\Phi_1(U)= -\partial_x U_2,\\
		&\Phi_2(U)= -R_0\partial_x \left[\varphi (U_1,U_2)\right] + \Phi_3(U) 
		\Big( e^{\frac{x+r}{\widetilde{\mu}}} \chi_- + \tfrac{\sinh(\frac{\ell-x}{\widetilde{\mu}})}{\sinh(\frac{\ell-r}{\widetilde{\mu}})}\chi_+\Big), \\
		&\Phi_3(U) = \dfrac{1}{	\alpha + \mathcal{T}_{\widetilde{\mu}}(h(U_1))}  \Big( \frac{ \eps}{2}\left \llbracket\frac{1}{h^2(U_1)}\right\rrbracket  U_3^2- \left\llbracket \mathcal{F}(U_1,U_2) \right\rrbracket  - \dfrac{P_0}{\epsilon}U_4\Big),\\
		&\Phi_4(U)  =- \kappa U_4 +\dfrac{\epsilon}{|\mathcal{E}_+|}U_3.
	\end{align*}
	
	By standard elliptic regularity theory, we know that $R_0$ in \eqref{R0} is a bounded operator from $H^m (\mathcal{E})$ to $H^{m+2}\cap H^1_0(\mathcal{E})$ for $m\geq -1$, thus $R_0\partial_x$ is a bounded operator from $H^m (\mathcal{E})$ to $H^{m+1} \cap H^1_0(\mathcal{E})$ for $m\geq 0$. Let us denote by
	$\mathcal{M}$ the open subset in $H^m(\mathcal{E})\times H^{m+1}(\mathcal{E})\times \mathbb{R}^2$ such that $\inf_{\mathcal{E}} h>0$. Then, 
	after recalling the definitions of $\varphi$ and $\mathcal{F}$ in \eqref{momentum-flux} and \eqref{calF}, 
	Sobolev trace and product estimates imply that $\Phi$ is a smooth map from $\mathcal{M}$ to $H^m(\mathcal{E})\times H^{m+1}(\mathcal{E})\times \mathbb{R}^2$ for $m\geq 1$. Therefore, applying Cauchy-Lipschitz theorem yields the existence of $T>0$ and a unique solution $U\in C^1([0,T); H^m(\mathcal{E})\times H^{m+1}(\mathcal{E})\times \mathbb{R}^2)$ to \eqref{compact-ODE}. In addition, exploiting the fourth component of the vectorial equation in \eqref{compact-ODE}, we have that $P_{\rm ch}$ belongs to $C^2([0, T); \mathbb{R})$. 
	Thanks to the assumption on $q^{\rm in}$ and  $q_i^{\rm in}$ in \eqref{initial-ass}, we have the equivalence 
	$$\llbracket q\rrbracket=0,  \quad \langle  q\rangle=q_i\quad \iff\quad \frac{d \llbracket q\rrbracket}{dt} =0,  \qquad   \dfrac{d\langle q\rangle}{dt}=\dfrac{dq_i}{dt} $$ and the regularity of the solution together with the second equation in \eqref{ODE-0} guarantees that the transmission conditions \eqref{trans-conditions} are satisfied on $[0,T)$. 
	Finally, coming back from \eqref{ODE-system-disp} to \eqref{added-evo-qi} concludes the proof of existence and uniqueness. Using more refined Moser product estimates, it can be shown that 
	\begin{equation*}
		\|\Phi (U)\|_{H^m(\mathcal{E})\times H^{m+1}(\mathcal{E})\times \mathbb{R}^2}\leq C_{\widetilde{\mu}}\Big( \Big\|U_1, U_2, \frac{1}{h(U_1)}\Big\|_{L^\infty(\mathcal{E})}, |U_4|
		\Big)\|U\|_{H^m(\mathcal{E})\times H^{m+1}(\mathcal{E})\times \mathbb{R}^2}
	\end{equation*}where $C_{\widetilde{\mu}} $ is a smooth non-decreasing function of its arguments. Therefore, by Grönwall's inequality, the solution to \eqref{compact-ODE} satisfies
	\begin{equation*}
		\|U(t)\|_{H^m(\mathcal{E})\times H^{m+1}(\mathcal{E})\times \mathbb{R}^2} \leq e ^{\int_0^t m(s)ds}  \|U^{\rm in}\|_{H^m(\mathcal{E})\times H^{m+1}(\mathcal{E})\times \mathbb{R}^2}
	\end{equation*}with $$m(s)=C_{\widetilde{\mu}}\Big( \Big\| U_1(s), U_2(s), \frac{1}{h(U_1(s))}\Big\|_{L^\infty(\mathcal{E})}, |U_4(s)|\Big). $$ 
	Thus, if the maximal existence time $T_{\rm max}$ is finite, 
	then necessarily one of the arguments of $C_{\widetilde{\mu}}$ 
	must blow up as $t\rightarrow T^-_{\rm max}$, otherwise the solution can be continued in $[0, T_{\rm max} + \tau)$ for some $\tau>0$ and we find a contradiction.
\end{proof}

\subsection{Toward wave-spring-mass systems} We now discuss the connection between the transmission problems we have derived in Subsection \ref{subsec-transmission} and the wave-spring-mass systems that will be derived in Section \ref{sec-springmass}.
Indeed, as a direct consequence of  \eqref{mean-qi}, the evolution equation \eqref{added-evo-qi} for the discharge in the interior domain can be formulated as an evolution equation for the mean surface elevation in the chamber.
\begin{proposition}The nonlinear first-order ODEs in Proposition \ref{prop-evo-qi} can be reformulated as  the following nonlinear second-order integro-differential equations:
	for $\widetilde{\mu}=0$,
	\begin{equation}\label{ODE-mean-0}
		\alpha |\mathcal{E}_+|\frac{d^2 \overline{\zeta}}{dt^2} + P_0\left(I - \kappa \mathcal{D}\right)\overline{\zeta} + \eps \frac{|\mathcal{E}_+|^2}{2}\left\llbracket\frac{1}{h^2}\right\rrbracket \Big(\frac{d\overline \zeta}{dt}\Big)^2 = - \llbracket \zeta \rrbracket,
	\end{equation}
	with $\mathcal{D}$ as in  \eqref{delay}; 
	for $\widetilde{\mu}\neq0$,
	\begin{equation}\label{ODE-mean}
		\left(\alpha + \mathcal{T}_{\widetilde{\mu}}(h)\right)	|\mathcal{E}_+|\frac{d^2\overline \zeta}{dt^2} + P_0 \left(I-\kappa\mathcal{D}\right)\overline\zeta -\eps\frac{|\mathcal{E}_+|^2}{2}\left\llbracket\frac{1}{h^2}\right\rrbracket \Big(\frac{d\overline \zeta}{dt}\Big)^2 =- \left\llbracket \mathcal{F}\right\rrbracket, 
	\end{equation}
	with $\mathcal{T}_{\widetilde{\mu}}(h)$ and $\mathcal{F}$ as in \eqref{T-mu}-\eqref{calF}.
\end{proposition}
The constrained shallow water models \eqref{dimless-shallow-set}-\eqref{contact-con} and \eqref{pressure-mean} can therefore be interpreted as transmission problems between the open water and chamber domains, with transmission conditions determined by the evolution of the mean surface elevation in the chamber. The dynamics \eqref{ODE-mean} is a nonlinear harmonic oscillator with a delay term, driven by the jump of $\zeta$ or $\mathcal{F}$ (according to wheter $\widetilde{\mu}=0$ or $\widetilde{\mu}\neq 0$) across the contact points. This second-order equation explicitly manifests the oscillatory nature of the fluid motion within the chamber, which is the main physical mechanism governing the OWC device, together with nonlinear effects inherited from the wave motion and a delay due to structural damping at the turbine.\\
The derivation of \eqref{ODE-mean} naturally motivates a different approach in the modeling of the dynamics of OWCs, used in ocean engineering \cite{Wang02, Reza18}, where the upper part of the fluid in the chamber is treated as a rigid layer free to move above the lower fluid and having the mean surface elevation $\overline{\zeta}$ as its degree of freedom. The goal of the next section is then to reformulate the shallow water models \eqref{dimless-shallow-set} in the presence of an OWC within this new modeling framework and to derive a second-order equation for the motion of the rigid layer analogous to \eqref{ODE-mean}.

\section{Wave-spring-mass systems}\label{sec-springmass}
In this section, we adopt a different modeling approach for the dynamics of OWCs in shallow water. Following the formulation used in \cite{Wang02, Reza18}, we model the fluid in the chamber as a multiphase system: the upper part is treated as a rigid layer, which we refer to as the \emph{water column}, that is free to move vertically above the lower fluid, see Figure \ref{owc-mass}.

\begin{figure}
	\includegraphics[scale=0.6]{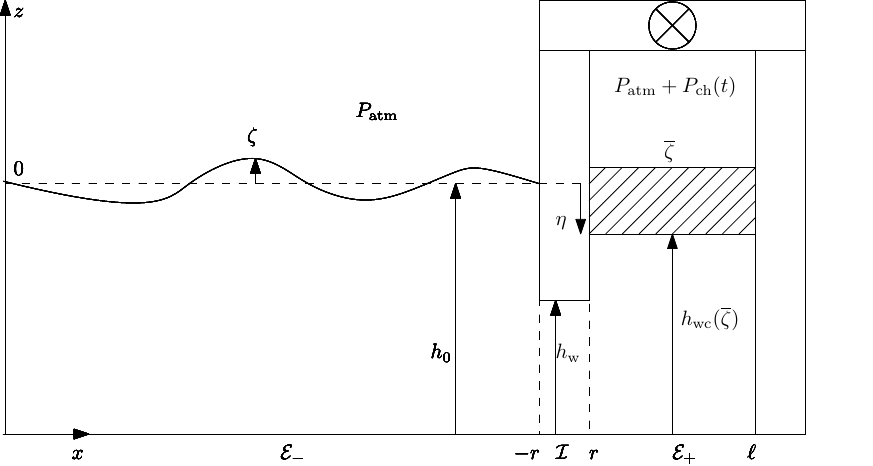}
	\caption{Configuration of the OWC in the new approach: inside the chamber, the rigid water column vertically moves above the lower fluid.}
	\label{owc-mass}
\end{figure}
The consequence of this change of modeling is twofold. First, inside the chamber, the surface elevation is replaced as an unknown by the elevation of the interface separating the two phases, which is assumed to be parametrized by a function $\eta(t,x)$. Therefore, instead of \eqref{dimless-shallow-set}, in the chamber domain we consider the dimensionless equations
\begin{equation}\label{inter-system}
	\begin{cases}
		\partial_t \eta + \partial_x q=0 \\[5pt]
		\Big(1-\widetilde{\mu}^2\partial_x^2\Big) \partial_t q + \eps \partial_x \Big(\dfrac{q^2}{\frak h}\Big) + \frak{h} \partial_x \eta=-\dfrac{\frak{h}}{\eps}\partial_x \underline{\mathcal{P}}
	\end{cases}\quad \text{in}\quad \mathcal{E}_+,
\end{equation} 
where $\frak{h}=1 + \eps\eta$ is the dimensionless fluid height below the interface 
and $\underline{\mathcal{P}}$ is the dimensionless interface pressure deviation from atmospheric pressure.
Second, the presence of the rigid water column introduces congestion into \eqref{inter-system} through a contact constraint analogous to \eqref{contact-con}. More precisely, we assume that the interface remains in contact with the bottom of the water column throughout the motion. Since the bottom is assumed to be flat and denoting by $\eta_{\rm eq}$ its vertical position at the equilibrium, it can be parametrized by the time-dependent function $ \eta_{\rm eq}+\overline{\zeta}$. Here, the mean surface elevation $\overline{\zeta}$ is no longer determined by the fluid equations through \eqref{mean-sur} but is instead an independent unknown representing the degree of freedom of the water column motion. The dimensionless contact constraint on the interface then reads
\begin{equation}\label{wc-con}
	\eta=   \eta_{\rm eq}+\overline{\zeta}  \qquad \text{in}\quad  \mathcal{E}_+.
\end{equation}

Thus, we are led to study the constrained shallow water models
\begin{equation}\label{NSW-open-int}
	\begin{cases}
		\partial_t \zeta + \partial_x q=0 \\[5pt]
		\Big(1-\widetilde{\mu}^2\partial_x^2\Big) \partial_t q + \eps \partial_x \Big(\dfrac{q^2}{h}\Big) + h \partial_x \zeta=  -\dfrac{h}{\eps}\partial_x{\underline{P}}\\[5pt]
	\end{cases} \quad\text{in} \quad\mathcal{E}_-\cup \mathcal{I}
\end{equation}
and  \begin{equation}\label{NSW-ch}
	\begin{cases}
		\partial_t \eta + \partial_x q=0 \\[5pt]
		\Big(1-\widetilde{\mu}^2\partial_x^2\Big) \partial_t q + \eps \partial_x \Big(\dfrac{q^2}{\frak{h}}\Big) + \frak{h} \partial_x \eta=-\dfrac{\frak{h}}{\eps}\partial_x \underline{\mathcal{P}}
	\end{cases}\quad \text{in}\quad \mathcal{E}_+
\end{equation} subject to the constraints
\begin{equation}\label{constraints}
	\underline{P}=0 \quad \text{in} \quad \mathcal{E}_-, \qquad \zeta=\zeta_{\rm w}\quad \text{in} \quad \mathcal{I},\qquad \eta=\eta_{\rm eq} +\overline\zeta \quad \text{in} \quad \mathcal{E}_+
\end{equation} and complemented by the far-field and wall conditions \eqref{bound-cond}. Analogously to \eqref{energy-balance}, the systems \eqref{NSW-open-int}-\eqref{NSW-ch} satisfy the local energy balances 
\begin{equation}\label{energy-balances}\begin{aligned}
		&	\partial_t e + \partial_x f = \frac{\underline{P} \partial_x q  }{\eps}+3\eps\widetilde{\mu}^2 R, \qquad &\text{in}\quad \mathcal{E}_-\cup \mathcal{I},\\[5pt]
		&	\partial_t \mathfrak{e} + \partial_x \mathfrak{f}  = \frac{\underline{\mathcal{P}} \partial_x q  }{\eps}+3\eps\widetilde{\mu}^2 	\mathfrak{R}, \qquad &\text{in}\quad \mathcal{E}_+,
	\end{aligned}
\end{equation}where $e(\zeta, q), f(\zeta, q, \underline{P}), R(\zeta, q)$ are as in \eqref{energy-density}-\eqref{R}, the fluid energy density and flux in the chamber domain are given by
\begin{equation}\label{energy-den-cham}
	\mathfrak{e}(\eta, q)= 	\mathfrak{e}_{\rm gra}(\eta)+ 	\mathfrak{e}_{\rm kin}(\eta, q)=\frac{\eta^2}{2}  + \frac{q^2}{2\frak{h}} + \widetilde{\mu}^2 \frac{(\partial_x q)^2}{2\frak{h}}
\end{equation}and 
\begin{equation}\label{frakflux}
	\mathfrak{f}(\eta, q, \underline{\mathcal{P}})= q\Big( \eta + \frac{\underline{\mathcal{P}}}{\eps}+ \eps\frac{q^2}{2\frak{h}^2} -\widetilde{\mu}^2\frac{\partial_x \partial_t q}{\frak{h}}  \Big),
\end{equation}
and the remainder term is 
\begin{equation}\label{frakR}
	\mathfrak{R}(\eta, q)=\frac{(\partial_x q)^3}{6\frak{h}^2}+\frac{q\partial_x \eta\partial_t \partial_x q}{3\frak{h}^2}.
\end{equation}
The fluid energy related to \eqref{NSW-open-int}-\eqref{NSW-ch} is defined by 
\begin{equation}\label{fluid-energy2}
	E_{\rm flu}(\zeta, \eta, q)= \int_{\mathcal{E}_-\cup\ \mathcal{I}}e (\zeta,q) + \int_{\mathcal{E}_+}\mathfrak{e} (\eta,q).
\end{equation}

\begin{proposition}\label{prop-model3}
The constrained shallow water models \eqref{NSW-open-int}-\eqref{constraints} can be reformulated as the following systems in the open water, interior and chamber domains: 
	\begin{equation}\label{eq-open}
		\begin{cases}
			\partial_t \zeta + \partial_x q=0 \\[5pt]
			\Big(1-\widetilde{\mu}^2\partial_x^2\Big) \partial_t q + \eps \partial_x \Big(\dfrac{q^2}{h}\Big) + h \partial_x \zeta=  0\\[5pt]
		\end{cases} \quad\text{in} \quad\mathcal{E}_-,
	\end{equation}
	\begin{equation}\label{eq-int}
		\begin{cases}
			q=q_i(t), \quad \zeta=\zeta_{\rm w} \\[5pt]
			\dfrac{dq_i}{dt}=  -\dfrac{h_{\rm w}}{\eps}\partial_x{\underline{P}}\\[5pt]
		\end{cases} \quad \text{in} \quad \mathcal{I},
	\end{equation}
	with $h_{\rm w}=1 + \eps\zeta_{\rm w}$, and
	\begin{equation}\label{eq-chamber}
		\begin{cases}
			q= \dfrac{d\overline \zeta}{dt}(\ell - x), \quad \eta= \eta_{\rm eq}+ \overline{\zeta}\\[7pt]
			(\ell - x)	\Big(\dfrac{d^2\overline \zeta}{dt^2} -\dfrac{2\eps}{h_{\rm wc}}  \Big(\dfrac{d\overline \zeta}{dt}\Big)^2 \Big)= - \dfrac{h_{\rm wc}}{\eps} \partial_x \underline{\mathcal{P}}
		\end{cases}\quad \text{in}\quad \mathcal{E}_+,
	\end{equation}with $h_{\rm wc}(\overline\zeta)= 1 +\eps\eta_{\rm eq} + \eps\overline{\zeta}$.
\end{proposition}

\begin{proof}Substituting \eqref{constraints} into \eqref{NSW-open-int}, we obtain the systems in the open water and interior domains, analogously to the derivation of \eqref{NSW-exterior} and \eqref{NSW-interior}. For the chamber domain, we use  \eqref{wc-con}, together with the wall condition in \eqref{bound-cond}, to deduce from \eqref{NSW-ch} the system
	\begin{equation}\label{NSW-chamber}
		\begin{cases}
			q= \dfrac{d\overline \zeta}{dt}(\ell - x)\\[7pt]
			(\ell - x)	\Big(\dfrac{d^2\overline \zeta}{dt^2} -\dfrac{2\eps}{h_{\rm wc}}  \Big(\dfrac{d\overline \zeta}{dt}\Big)^2 \Big)= - \dfrac{h_{\rm wc}}{\eps} \partial_x \underline{\mathcal{P}}
		\end{cases}\quad \text{in}\quad \mathcal{E}_+,
	\end{equation}
	where $h_{\rm wc}(\overline\zeta)= 1 +\eps\eta_{\rm eq} + \eps\overline{\zeta}$.
Two similarities can be observed between the fluid dynamics in the chamber domain and in the interior domain. First, in the momentum equation of \eqref{NSW-chamber} the dispersive term disappears since the discharge is linear with respect to $x$. Second, the source term in the momentum equation does not vanish, unlike in \eqref{NSW-exterior}, because the analogue of the surface pressure constraint \eqref{surpres-ext} does not hold for the interface pressure $\underline{\mathcal{P}}$. Indeed, it is an unknown that depends on the motion of the water column and acts as a Lagrange multiplier associated with the contact constraint \eqref{wc-con}. 
\end{proof}

\subsection{Water column dynamics}In order to close the systems derived in Proposition \ref{prop-model3}, we need to determine the evolution of the mean surface elevation $\overline{\zeta}$ in the chamber domain, which governs the water column dynamics. To this end, we assume that the water column moves with uniform vertical velocity. Therefore, its motion obeys Newton's second law
\begin{equation}\label{newton0}
	m_{\rm wc} \frac{d^2 \overline\zeta}{dt^2} = - m_{\rm wc} g + \int_{\mathcal{E}_+}\left( \underline{\mathcal{P}} - P_{\rm ch}\right)
\end{equation}where $m_{\rm wc}$ is the water column mass, $\underline{\mathcal{P}}$ is the interface pressure deviation in the chamber and $P_{\rm ch}$ is the air pressure deviation. 
Passing to the dimensionless form (see Appendix \ref{appsec-newton}) and using \eqref{Pch-dimless}, \eqref{newton0} becomes
\begin{equation}\label{newton2}
	T_{\rm buo}^2 \frac{d^2\overline\zeta}{dt^2} = -\frac{m_{\rm wc}}{\eps} - P_0\left(I-\kappa \mathcal{D} \right)\overline\zeta+ \frac{1}{\eps |\mathcal{E}_+|}\int_{\mathcal{E}_+} \underline{\mathcal{P}}
\end{equation}
where $P_0$ is as in \eqref{pressure-mean} and $2\pi T_{\rm buo}$ is the dimensionless hydrostatic buoyancy period of the motion defined by 
\begin{equation}\label{buo-period}
	T^2_{\rm buo}= \frac{m_{\rm wc} h_0^2}{ L^2}.
\end{equation}
The notion of hydrostatic buoyancy period naturally arises from an equivalent formulation of \eqref{newton2}. Indeed, combining the dimensionless Archimedes' principle
\begin{equation}\label{archimede}
	m_{\rm wc}=-\eps\eta_{\rm eq}
\end{equation} with the dimensionless version of \eqref{wc-con}, we obtain that
\begin{equation}\label{newton0-ref}
	T^2_{\rm buo} \frac{d^2 \overline\zeta}{dt^2} = - \overline\zeta - P_0\left(I-\kappa \mathcal{D} \right)\overline\zeta+ \eta + \frac{1}{\eps|\mathcal{E}_+|} \int_{\mathcal{E}_+} \underline{\mathcal{P}} .
\end{equation} 
In particular, the last two terms on the right-hand side are responsible for the coupling between the water column dynamics and the fluid dynamics in the chamber domain.	 
Associated with the motion of the water column, we introduce the mechanical energy given by the sum of the kinetic, gravitational and elastic energies. In dimensionless form, it reads
\begin{equation}\label{wc-energy}
	E_{\rm wc}(\overline{\zeta})= \frac{T_{\rm buo}^2|\mathcal{E}_+|}{2} \Big(\frac{d\overline \zeta}{dt}\Big)^2 + \frac{m_{\rm wc} |\mathcal{E}_+|}{\eps}\overline\zeta +  \frac{\tau}{2}\overline\zeta^2,
\end{equation} where the last term on the right-hand side coincides with the elastic energy \eqref{ela-energy-dimless}. Then, it follows from \eqref{newton2} that
\begin{equation}\label{der-Ewc}
	\frac{d }{dt}E_{\rm wc} =|\mathcal{E}_+| \frac{d\overline \zeta}{dt} \Big(T_{\rm buo}^2 \frac{d^2\overline\zeta}{dt^2} +\frac{m_{\rm wc}}{\eps} + P_0 \overline\zeta\Big)= \frac{d\overline \zeta}{dt} \Big(\tau \kappa \mathcal{D}\overline \zeta + \frac{1}{\eps} \int_{\mathcal{E}_+}\underline{\mathcal{P}}\Big).\bigskip
\end{equation}

\subsection{Coupled problems between the open water, interior and chamber domains}At this point, we have derived the three systems \eqref{eq-open}, \eqref{eq-int} and \eqref{eq-chamber}, which separately govern the fluid dynamics in the open water, interior and chamber domains, respectively.  As already seen in Section \ref{sec-constramodels}, these systems must be coupled through suitable coupling conditions in order to obtain a closed formulation. To this end, we consider again Assumption \ref{assumption} but with a different definition of the total energy of the system that takes into account the new modeling in the OWC chamber. We define it as the sum of the fluid energy and the mechanical energy of the water column,
\begin{equation}\label{total-energy2}
	\widetilde{E}_{\rm tot}(\zeta,\eta, q, \overline\zeta )= E_{\rm flu}(\zeta, \eta, q) + E_{\rm wc}(\overline\zeta)
\end{equation} where $E_{\rm flu}$ and $E_{\rm wc}$ are respectively as in \eqref{fluid-energy2} and \eqref{wc-energy}. Thus, we introduce the next assumption:

\begin{assumption}\label{assumption2}The following properties hold:
	\begin{equation}
		\tag{i}  \label{discharge-con2}
		q_{|_{\mathcal{E}_-}} =q_i\quad \text{at} \quad x=-r, \qquad 	q_{|_{\mathcal{E}_+}} =q_i\quad \text{at} \quad x=r;
	\end{equation}
	\begin{equation}\label{tot-ene-conservation2} 
		\tag{ii} 
		\text{when $\kappa=0$ in \eqref{Pch-dimless},} \quad  
		\frac{d}{dt}\widetilde{E}_{\rm tot} = 3\eps\widetilde{\mu}^2 \Big(\int_{\mathcal{E}_- \cup \ \mathcal{I}} R + \int_{\mathcal{E}_+} \mathfrak{R}\Big)= O(\eps\mu),
	\end{equation} with $R$ and $\mathfrak{R}$ respectively as in \eqref{R} and \eqref{frakR}. 
	
\end{assumption}
Analogously to Proposition \ref{prop-conserve-coup} we derive an explicit coupling condition at the contact points starting from item \eqref{tot-ene-conservation2} of Assumption \ref{assumption2}.
\begin{proposition}
	Consider regular solutions $(\zeta, \eta, q, \overline\zeta)$ to the fluid-water column systems \eqref{NSW-open-int}-\eqref{constraints} and \eqref{newton0-ref} with $\kappa=0$. Then, item \eqref{tot-ene-conservation2}  in Assumption \ref{assumption2} holds if and only if 
	\begin{equation}\label{f-con2}
		\left\llbracket f\right\rrbracket^*_{-r}= \mathfrak{f}(r) - f_{|_{\mathcal{I}}}(r),
	\end{equation} 
	with $f$ and $\mathfrak{f}$ respectively as in \eqref{flux} and \eqref{frakflux}.
	In addition, using item \eqref{discharge-con2} of Assumption \ref{assumption2},  \eqref{f-con2} reduces to
	\begin{equation}\label{bra-con2}
		\left\llbracket \zeta + \frac{\underline{P}}{\eps}+ \mathfrak{E} \right\rrbracket^*_{-r}  =  \Big(\eta+\frac{\underline{\mathcal{P}}}{\eps} +	\widetilde{\mathfrak{E}}  \Big)(r)-\Big(\zeta_{|_{\mathcal{I}}} +\frac{\underline{P}_{|_{\mathcal{I}}}}{\eps}+\mathfrak{E}_{|_{\mathcal{I}}}\Big)(r)
	\end{equation}or, equivalently, 
	\begin{equation}\label{bra-con-equi2}
		\frac{\llbracket\underline{P}_{|_{\mathcal{I}}}\rrbracket}{\eps} =\Big( \eta +\frac{\underline{\mathcal{P}}}{\eps}+	\widetilde{\mathfrak{E}}\Big)(r)- \Big( \zeta_{|_{\mathcal{E}_-}}+ \frac{\underline{P}_{|_{\mathcal{E}_-}}}{\eps}+  \mathfrak{E}_{|_{\mathcal{E}_-}} \Big)(-r),
	\end{equation}
	with $\mathfrak{E}$ as in \eqref{frak-E}  and 
	\begin{equation}\label{frak-Etilde}
		\widetilde{\mathfrak{E}}(\eta, q)=  \eps  \frac{q^2}{2h^2}-\widetilde{\mu}^2\frac{\partial_x\partial_t q}{h}.
	\end{equation}
\end{proposition}
\begin{proof}
	Integrating the local energy balances \eqref{energy-balances} over $\mathcal{E}_-\cup \mathcal{I}$ and $\mathcal{E}_+$,respectively, and using the far-field and wall conditions \eqref{bound-cond} yields the global energy balance
	\begin{equation*}
		\frac{dE_{\rm flu} }{dt}  - \mathfrak{f}(r) + f_{|_{\mathcal{I}}}(r) +\llbracket f \rrbracket_{-r}^* = \int_{\mathcal{E}_- \cup \ \mathcal{I}} \frac{\underline{P}\partial_x q}{\eps} + \int_{\mathcal{E}_+} \frac{\underline{\mathcal{P}}\partial_x q}{\eps} + 3\eps\widetilde{\mu}^2 \Big( \int_{\mathcal{E}_- \cup\ \mathcal{I}} R + \int_{\mathcal{E}_+}\mathfrak{R} \Big).
	\end{equation*} In view of the pressure constraint in \eqref{constraints} and the equations in the interior and chamber domains \eqref{eq-int}-\eqref{eq-chamber}, we obtain that
	\begin{equation*}\begin{aligned}
			\frac{dE_{\rm flu}}{dt}   - \mathfrak{f} (r) + f_{|_{\mathcal{I}}}(r) +\llbracket f \rrbracket_{-r}^* &= - \frac{1}{\eps}\frac{d\overline\zeta}{dt}\int_{\mathcal{E}_+}\underline{\mathcal{P}} + 3\eps\widetilde{\mu}^2\Big(\int_{\mathcal{E}_-\cup\ \mathcal{I}} R +\int_{\mathcal{E}_+} \mathfrak{R}\Big)\\[5pt]
			&=-\dfrac{dE_{\rm wc}}{dt} + \tau \kappa \mathcal{D}\overline\zeta \dfrac{d\overline\zeta}{dt}+ 3\eps\widetilde{\mu}^2\Big(\int_{\mathcal{E}_-\cup\ \mathcal{I}} R +\int_{\mathcal{E}_+} \mathfrak{R}\Big),
		\end{aligned}
	\end{equation*}where in the second equality we have used \eqref{der-Ewc}. It follows from \eqref{total-energy2} that
	\begin{equation*}
		\frac{d \widetilde{E}_{\rm tot}}{dt}=  \mathfrak{f} (r) - f_{|_{\mathcal{I}}}(r) -\llbracket f \rrbracket_{-r}^* + \tau\kappa\mathcal{D}\overline\zeta \frac{d\overline\zeta}{dt}+ 3\eps\widetilde{\mu}^2\Big(\int_{\mathcal{E}_-\cup\ \mathcal{I}} R +\int_{\mathcal{E}_+} \mathfrak{R}\Big)
	\end{equation*} and therefore item \eqref{tot-ene-conservation2} of Assumption \ref{assumption2} holds if and only if 
	\begin{equation}\label{f-cond}
		\llbracket f \rrbracket_{-r}^*=	\mathfrak{f}(r) - f_{|_{\mathcal{I}}}(r).
	\end{equation}

	Resorting to the definitions of the energy fluxes \eqref{flux} and \eqref{frakflux}, and combining item \eqref{discharge-con2} of Assumption \ref{assumption2} with \eqref{eq-int}, we find \eqref{bra-con2}. Equivalently, this condition can be written as
	\begin{equation*}
		\left	\llbracket \zeta_{|_{\mathcal{I}}}+ \frac{\underline{P}_{|_{\mathcal{I}}}}{\eps} + \mathfrak{E}_{|_{\mathcal{I}}} \right\rrbracket= \Big( \eta +\frac{\underline{\mathcal{P}}}{\eps}+ \widetilde{\mathfrak{E}}  \Big)(r)- \Big( \zeta_{|_{\mathcal{E}_-}}+ \frac{\underline{P}_{|_{\mathcal{E}_-}}}{\eps}+  \mathfrak{E}_{|_{\mathcal{E}_-}} \Big)(-r),
	\end{equation*}
	so that \eqref{bra-con-equi2} follows from the fact that $\llbracket \zeta_{|_{\mathcal{I}}} + \mathfrak{E}_{|_{\mathcal{I}}} \rrbracket =0$.
\end{proof} 
The conditional result for conservation of the total energy at order $O(\eps\mu)$, analogous to Proposition \ref{prop-intpres-con}, follows directly from \eqref{frakE-red} and \eqref{constraints}.
\begin{proposition} \label{prop-intpres-con2}
	Considering the shallow water Bernoulli principles at the contact points
	\begin{align*}
		&\zeta_{\rm w}+\frac{\underline{P}_{|_{\mathcal{I}}}}{\eps} + \eps\frac{q_i^2}{2h_{\rm w}^2}= \zeta_{|_{\mathcal{E}_-}}  
		+\mathfrak{E}_{|_{\mathcal{E}_-}}  
		&\quad \text{at} \quad x=-r, \\[5pt]
		&\zeta_{\rm w} +\frac{\underline{P}_{|_{\mathcal{I}}}}{\eps}+ \eps\frac{q_i^2}{2h_{\rm w}^2}= \eta+\frac{\underline{\mathcal{P}}}{\eps}
		+\widetilde{\mathfrak{E}} 
		&\quad \text{at} \quad x= r,
	\end{align*}with ${\mathfrak{E}}$ and $\widetilde{\mathfrak{E}}$ respectively as in \eqref{frak-E} and \eqref{frak-Etilde},
	regular solutions to the shallow water models in the presence of a water column \eqref{eq-open}-\eqref{eq-chamber} and \eqref{newton0-ref} with $\kappa=0$ conserve the total energy \eqref{total-energy2} at order $O(\eps\mu)$. 
\end{proposition}

\subsection{Wave-spring-mass interaction problems}
The coupling conditions introduced in the previous subsection allow us to write the last two terms on the right-hand side of \eqref{newton0-ref}, which couple the water column dynamics with the fluid dynamics in the chamber domain, in terms of the fluid unknowns in the open water domain. As a consequence, the dynamics of OWCs can be reformulated as an interaction problem between spring-mass systems inside the chamber and shallow water waves outside it.

\begin{proposition}\label{prop-wavespringmass}
	Under Assumption \ref{assumption2}, solving the fluid-water column systems \eqref{NSW-open-int}-\eqref{constraints} and \eqref{newton0-ref} is equivalent to solving the following wave-spring-mass systems: 
	\begin{equation}\label{shallow-open}
		\begin{cases}
			\partial_t \zeta + \partial_x q=0 \\[5pt]
			\Big(1-\widetilde{\mu}^2\partial_x^2\Big) \partial_t q + \eps \partial_x \Big(\dfrac{q^2}{h}\Big) + h \partial_x \zeta=  0\\[5pt]
		\end{cases} \quad\text{in} \quad\mathcal{E}_-, \quad \widetilde{\mu}\geq 0,
	\end{equation} where $h=1+\eps\zeta$, together with the far-field and boundary conditions
	\begin{equation}\label{coupling-con}(\zeta, q)\rightarrow (0,0) \quad \text{as}\quad x\rightarrow -\infty, \qquad  q=\dfrac{d\overline{\zeta}}{dt}|\mathcal{E}_+| \quad \text{at}\quad x=-r,\end{equation}
	coupled with the following nonlinear second-order integro-differential equations: for $\widetilde{\mu}=0$,
	\begin{equation}\label{added-springmass-0}
		\mathcal{T}^2_{\rm buo}(h_{\rm wc})
		\frac{d^2 \overline\zeta}{dt^2} = - (1+P_0) \overline\zeta + \kappa P_0 \mathcal{D}\overline\zeta +\eps \beta(h, h_{\rm wc}) \Big(\frac{d\overline \zeta }{dt} \Big)^2  +\zeta(-r),
	\end{equation} 
	with 
	\begin{equation}\label{increased-buoyancy-0}
		\begin{aligned}
			\mathcal{T}^2_{\rm buo}(h_{\rm wc})=		T^2_{\rm buo} +  \alpha|\mathcal{E}_+| + \frac{|\mathcal{E}_+|^2}{3 h_{\rm wc}} , \ \quad \beta(h, h_{\rm wc})= \frac{|\mathcal{E}_+|^2}{2}\Big( \frac{1}{h^2(-r)}+ \frac{1}{3h^2_{\rm wc}}\Big),	
		\end{aligned}
	\end{equation}
	and $h_{\rm wc}(\overline\zeta)=1+\eps\eta_{\rm eq} + \eps\overline\zeta$;
	for $\widetilde{\mu}\neq 0$,
	\begin{equation}\label{added-springmass}
		\mathcal{T}^2_{\widetilde{\mu},\rm buo}(h,h_{\rm wc})
		\frac{d^2 \overline\zeta}{dt^2} = - (1+P_0) \overline\zeta + \kappa P_0 \mathcal{D}\overline\zeta +\eps\widetilde{\beta}(h, h_{\rm wc}) \Big(\frac{d\overline \zeta }{dt} \Big)^2  +\mathcal{F}(-r),
	\end{equation} 
	with $\mathcal{F}$ as in \eqref{calF}, and 
	\begin{equation}\label{increased-buoyancy}
		\begin{aligned}
			\mathcal{T}^2_{\widetilde{\mu},\rm buo}(h,h_{\rm wc})\!=\!	\mathcal{T}^2_{\rm buo}(h_{\rm wc})\!+ \!\frac{\widetilde{\mu}|\mathcal{E}_+| }{h(-r)}\! +\frac{\widetilde{\mu}^2}{h_{\rm wc}}, \ \ \
			\widetilde{\beta}(h, h_{\rm wc})= \frac{|\mathcal{E}_+|^2}{2}\Big(\! \!-\!\frac{1}{h^2(-r)}\!+\! \frac{1}{3h^2_{\rm wc}}\Big).
		\end{aligned} 
	\end{equation}
	Moreover, the discharge in the interior domain is given by
	\begin{equation}\label{q-interior}
		q= \dfrac{d\overline{\zeta}}{dt}|\mathcal{E}_+|  \quad \text{in}\quad \mathcal{I},
	\end{equation} while the interface elevation and discharge in the chamber domain are given by
	\begin{equation}\label{eta-q-chamber}
		\eta = \eta_{\rm eq} +\overline{\zeta}, \qquad q=\dfrac{d\overline{\zeta}}{dt}(\ell-x) \quad \text{in}\quad \mathcal{E}_+.
	\end{equation}
	
\end{proposition}
\begin{proof}
	On the one hand, combining \eqref{eq-int}-\eqref{eq-chamber} with item \eqref{discharge-con2} of Assumption \ref{assumption2}, we obtain that
	\begin{equation}\label{q-cond}
		q_{|_{\mathcal{E}_-}}=q_i= \dfrac{d\overline \zeta}{dt}|\mathcal{E}_+| \quad \text{at} \quad x=-r,
	\end{equation} which implies \eqref{q-interior}, whereas \eqref{eta-q-chamber} is exactly the first line in \eqref{eq-chamber}. 
	In addition, integrating the momentum equation in \eqref{eq-int} over $\mathcal{I}$ and replacing $q_i$ by \eqref{q-cond} yields that 
	\begin{equation}\label{Pint-bra}
		\frac{\llbracket \underline{P}_{|_{\mathcal{I}}}\rrbracket}{\eps}= -\alpha|\mathcal{E}_+|\frac{d^2\overline \zeta}{dt^2}
	\end{equation}with $\alpha=2r/h_{\rm w}$ and, in view of \eqref{eq-chamber}, \eqref{frak-Etilde} reads 
	\begin{equation}\label{frakEtilde-exp}
		\widetilde{\mathfrak{E}}(r)=    \frac{\eps|\mathcal{E}_+|^2}{2h_{\rm wc}^2}\Big( \dfrac{d \overline \zeta}{dt}\Big)^2 +\frac{\widetilde{\mu}^2}{h_{\rm wc}}\dfrac{d^2 \overline \zeta}{dt^2}.
	\end{equation}
	On the other hand, we write the last two terms in \eqref{newton0-ref} in terms of $\overline\zeta$ and the fluid unknowns in the open water domain. By means of integration by parts and using the momentum equation in \eqref{eq-chamber}, we write
	\begin{equation}\label{coupling-wc-eq}
		\begin{aligned}
			\frac{1}{\eps |\mathcal{E}_+|}	\int_{\mathcal{E}_+}\underline{\mathcal{P}}&=  \frac{\underline{\mathcal{P}}(r)}{\eps} -  \frac{1}{\eps |\mathcal{E}_+|}\int_{\mathcal{E}_+} (x-\ell)\partial_x\underline{\mathcal{P}} 
			\\[5pt]&= \frac{\underline{\mathcal{P}}(r)}{\eps} - \frac{|\mathcal{E}_+|^2}{3 h_{\rm wc}} \dfrac{d^2\overline\zeta}{dt^2} +\frac{2\eps|\mathcal{E}_+|^2}{3h^2_{\rm wc}}\Big(\dfrac{d\overline\zeta}{dt} \Big)^2,
		\end{aligned}
	\end{equation}while combining \eqref{bra-con-equi2}, \eqref{constraints} and \eqref{Pint-bra}-\eqref{frakEtilde-exp} yields that 
	\begin{equation}\label{eta-P-trace}
		\Big(\eta + \frac{\underline{\mathcal{P}}}{\eps}\Big) (r) = \left(\zeta +\mathfrak{E}\right)(-r)  - \frac{\eps|\mathcal{E}_+|^2}{2h_{\rm wc}^2}\Big( \dfrac{d \overline \zeta}{dt}\Big)^2-\Big(\alpha |\mathcal{E}_+|  +\frac{\widetilde{\mu}^2}{h_{\rm wc}}\Big) \frac{d^2\overline\zeta}{dt^2}.
	\end{equation}
	For $\widetilde{\mu}=0$, since $\mathfrak{E}$ reads as in \eqref{frakE-red}, it follows from \eqref{q-cond} that 
	\begin{equation}\label{frakE-red-trace}
		\mathfrak{E}(-r) = \eps \frac{|\mathcal{E}_+|^2}{2h^2(-r)}\Big(\frac{d\overline\zeta}{dt} \Big)^2.
	\end{equation}Therefore, gathering \eqref{coupling-wc-eq}-\eqref{frakE-red-trace} together implies that
	\begin{equation*}
		\eta + \frac{1}{\eps |\mathcal{E}_+|}	\int_{\mathcal{E}_+}\underline{\mathcal{P}} = \zeta (-r) + \frac{\eps |\mathcal{E}_+|^2}{2}\Big( \frac{1}{h^2(-r)} + \frac{1}{3h^2_{\rm wc}} \Big) \Big( \frac{d\overline \zeta}{dt}\Big)^2 - \Big( \alpha|\mathcal{E}_+| + \frac{|\mathcal{E}_+|^2}{3 h_{\rm wc}} \Big)\frac{d^2\overline \zeta}{dt^2}
	\end{equation*}
	and injecting this expression into \eqref{newton0-ref} gives \eqref{added-springmass-0}.
	For $\widetilde{\mu}\neq 0$, we have that
	\begin{equation} \label{frakE-trace}
		\mathfrak{E}(-r)= \eps \frac{|\mathcal{E}_+|^2}{2h^2(-r)}\Big(\frac{d\overline\zeta}{dt} \Big)^2 -\widetilde{\mu}^2 \frac{\partial_x\partial_t q}{h}(-r).
	\end{equation}
	Resorting to the proof of Proposition \ref{prop-evo-qi} and using \eqref{q-cond}, we obtain that 
	\begin{equation*}
		\partial_t q = -R_0\partial_x \varphi + |\mathcal{E}_+|\frac{d^2\overline \zeta}{dt^2} e^{\frac{x+r}{\widetilde{\mu}}}\quad \text{in} \quad \mathcal{E}_-
	\end{equation*}with $\varphi$ as in \eqref{momentum-flux} and, in view of \eqref{R0-R1}, we write
	\begin{equation}\label{disp-openwater}
		-\widetilde{\mu}^2 \frac{\partial_x\partial_t q}{h}= \dfrac{(R_1-I)\varphi}{h} - \widetilde{\mu} |\mathcal{E}_+| \frac{d^2 \overline \zeta}{dt^2}\frac{e^{\frac{x+r}{\widetilde{\mu}}}}{h} \quad \text{in} \quad \mathcal{E}_-.
	\end{equation}After using the explicit expression of $\varphi$ and again \eqref{q-cond}, we combine \eqref{frakE-trace} with \eqref{disp-openwater} to obtain that
	\begin{equation}\label{zeta-frakE-trace}
		\left(\zeta + \mathfrak{E} \right)(-r)= \mathcal{F}(-r)-\frac{\eps |\mathcal{E}_+|^2}{2h^2(-r)}\Big(\frac{d\overline\zeta}{dt} \Big)^2 -\frac{\widetilde{\mu}|\mathcal{E}_+| }{h(-r)}\frac{d^2 \overline \zeta}{dt^2},
	\end{equation}with $\mathcal{F}$ as in \eqref{calF}. Therefore, gathering \eqref{coupling-wc-eq}-\eqref{eta-P-trace} and \eqref{zeta-frakE-trace} together yields that
	\begin{align*}
		\eta + \frac{1}{\eps |\mathcal{E}_+|}	\int_{\mathcal{E}_+}\underline{\mathcal{P}} =\  &\mathcal{F} (-r) + \frac{\eps |\mathcal{E}_+|^2}{2}\Big( -\frac{1}{h^2(-r)} + \frac{1}{3h^2_{\rm wc}} \Big) \Big( \frac{d\overline \zeta}{dt}\Big)^2 \\[2pt]&- \Big( \alpha|\mathcal{E}_+| + \frac{|\mathcal{E}_+|^2}{3 h_{\rm wc}} + \frac{\widetilde{\mu}|\mathcal{E}_+| }{h(-r)} +\frac{\widetilde{\mu}^2}{h_{\rm wc}} \Big)\frac{d^2\overline \zeta}{dt^2}
	\end{align*}
	and substituting this expression into \eqref{newton0-ref} finally gives \eqref{added-springmass}.
\end{proof}

\paragraph{\textit{Effective buoyancy period: added mass vs spring force}}The integro-differential equations \eqref{added-springmass-0} and \eqref{added-springmass} exhibit the so-called {added mass} phenomenon, which is typical of fluid-structure interaction problems. 
Indeed, the additional leading-order terms appearing on the left-hand side through the positive coefficients in $\mathcal{T}^2_{\rm buo}(h_{\rm wc})$ and $\mathcal{T}^2_{\widetilde{\mu},\rm buo}(h,h_{\rm wc})$ that increase $T^2_{\rm buo}$, describe the fact that, in order to move, the water column must accelerate not only itself but also the surrounding fluid, thereby increasing its effective mass. Moreover, as for \eqref{added-evo-qi}, retaining the new leading-order terms on the left-hand side of  \eqref{added-springmass-0} and \eqref{added-springmass} avoids numerical instabilities in simulations.
When $\widetilde{\mu}=0$, combining \eqref{buo-period} with \eqref{increased-buoyancy-0}, yields the dimensionless added mass
\begin{align}\label{dimless-addedmass-0}
	m_{\rm add} 
	=\frac{L^2}{h_0^2}\Big( 
	\alpha|\mathcal{E}_+| + \frac{|\mathcal{E}_+|^2}{3 h_{\rm wc}} \Big).
\end{align}
When $\widetilde{\mu}\neq 0$, dispersive effects further contribute to the added mass phenomenon and introduce a new coupling mechanism. On the one hand,
$\mathcal{T}^2_{\widetilde{\mu},\rm buo}(h,h_{\rm wc})$ contains two additional positive terms compared with $\mathcal{T}^2_{\rm buo}(h_{\rm wc})$. On the other hand, the additional term involving the trace of the fluid height at the left wall of the structure reveals an interesting feature: in the dispersive case, the added mass effect is directly coupled with the fluid dynamics in the open water domain. In view of \eqref{increased-buoyancy}, the dimensionless added mass then reads 
\begin{align}\label{dimless-addedmass}
	m_{\rm add} 
	=\frac{L^2}{h_0^2}\Big( 
	\alpha|\mathcal{E}_+| + \frac{|\mathcal{E}_+|^2}{3 h_{\rm wc}} \Big) +  \frac{L|\mathcal{E}_+|}{h_0\sqrt{3}h(-r)} +\frac{1}{3h_{\rm wc}}.
\end{align}
A second difference between the hyperbolic and dispersive cases concerns the nature of the last term on the right-hand side of  \eqref{added-springmass-0} and \eqref{added-springmass}, which provides the driving force for the water column motion. In \eqref{added-springmass-0}, this term is local, as it is given by the trace of the fluid surface elevation at the left structure wall. In \eqref{added-springmass}, instead, it is non-local since computing the trace of $\mathcal{F}$, which contains the operator $R_1$, requires the knowledge of the fluid unknowns over the entire open water domain, see \eqref{R1}.

Another important feature revealed by \eqref{added-springmass-0} and \eqref{added-springmass} is related to the presence of the spring force $-P_0 \overline\zeta$, which accounts for the restoring action of the air pressure variation in the chamber. In fact, the effective buoyancy period of the water column motion is determined by the competition between the added mass phenomenon and the stiffness of the spring force. It is indeed given by $2\pi\mathfrak T_{\rm buo}$, where
\begin{equation}\label{final-buoyancy}
	\mathfrak{T}_{\rm buo}^2= \frac{\mathcal{T}^2_{\rm buo}(h_{\rm wc})}{1+P_0}  \quad \text{for} \ \widetilde{\mu}=0\qquad \text{or} \qquad \mathfrak{T}_{\rm buo}^2= \frac{\mathcal{T}^2_{\widetilde{\mu},\rm buo}(h,h_{\rm wc})}{1+P_0}  \quad \text{for} \ \widetilde{\mu}\neq0.
\end{equation} 
In the absence of air pressure variation, as in the case of floating buoys \cite{BecLan22}, one has $P_0\equiv 0$ and the hydrostatic buoyancy period $2\pi T_{\rm buo}$ is increased by the interaction with the fluid, with a larger increase when dispersive effects are taken into account. For OWCs, however, the situation may be different as $P_0$ does not vanish. In particular, combining \eqref{buo-period} with \eqref{increased-buoyancy-0} and \eqref{increased-buoyancy}, we find three different scenarios: 

\begin{proposition}
	Let $P_0$ be as in \eqref{pressure-mean}, $m_{\rm wc}$ be the dimensionless water column mass and $m_{\rm add}$ be the dimensionless added mass \eqref{dimless-addedmass-0} when $\widetilde{\mu}=0$ or \eqref{dimless-addedmass} when $\widetilde{\mu}\neq0$. The following relations between the effective and hydrostatic buoyancy periods, defined respectively by \eqref{final-buoyancy} and \eqref{buo-period}, hold:
	\begin{enumerate}
		\item if $P_0 m_{\rm wc} < m_{\rm add}$ , then  $\mathfrak{T}_{\rm buo} > T_{\rm buo}$;\smallskip
		\item if $P_0 m_{\rm wc} = m_{\rm add}$ , then  $\mathfrak{T}_{\rm buo} = T_{\rm buo}$;\smallskip
		\item if $P_0 m_{\rm wc} >m_{\rm add}$ , then  $\mathfrak{T}_{\rm buo}< T_{\rm buo}$.\smallskip
	\end{enumerate}
	In particular, the effective buoyancy period is larger than the hydrostatic one for sufficiently large chamber heights and smaller for sufficiently small ones.
\end{proposition}

\subsection{Well-posedness of the initial boundary value problems} 
After having reformulated the shallow water models in the presence of a water column as wave-spring-mass systems \eqref{shallow-open}-\eqref{added-springmass}, we complement them with the initial conditions 
\begin{equation}\label{initial-conds}
	(\zeta, q)(0, \cdot )= (\zeta^{\rm in}, q^{\rm in}) \quad \text{in} \quad \mathcal{E}_-, \qquad \Big(\overline\zeta, \frac{d\overline\zeta}{dt}\Big)(0) = (\overline\zeta^{\rm in},\overline\zeta_1^{\rm in} ).
\end{equation}

\paragraph{{\emph{Hyperbolic case}}.} In the case $\widetilde{\mu}=0$, the wave-spring-mass system \eqref{shallow-open}-\eqref{added-springmass-0} is essentially equivalent to the wave-piston system studied in \cite[Section 4]{IguLan21}, where the authors investigated the motion of shallow water waves pushed by a lateral piston. In fact, the analysis here is simpler since the nonlinear shallow water equations are set in the time-independent domain $\mathcal{E}_-$, while in their setting the fluid domain has a moving boundary, determined by the horizontal position of the piston, and a suitable Lagrangian diffeomorphism is introduced. 

Seeking regular solutions requires compatibility conditions. To this end, we denote $(\zeta_k, q_k)=(\partial_t^k \zeta, \partial_t^k q) $ for $k\geq0$ and exploit the quasilinear structure of  \eqref{shallow-open} to write inductively $(\zeta_k, q_k)$ as polynomials of spatial derivatives of $(\zeta, q)$ up to order $k$. Therefore, after denoting $(\zeta_k^{\rm in}, q_k^{\rm in})= (\zeta_k, q_k) (0, \cdot)$ and recalling \eqref{initial-conds}, it follows that 
\begin{equation}\label{q-k-in}
	\zeta_k^{\rm in}= P_k(\zeta^{\rm in}, q^{\rm in}, \dots, \partial_x^k \zeta^{\rm in},  \partial_x^kq^{\rm in}), \quad
	q_k^{\rm in}= Q_k(\zeta^{\rm in}, q^{\rm in}, \dots, \partial_x^k \zeta^{\rm in},  \partial_x^kq^{\rm in}), \quad k\geq 0,
\end{equation} 
where $P_k, Q_k$ are polynomials of their arguments with each monomial containing at most $k$ derivatives of $(\zeta^{\rm in},q^{\rm in})$. Similarly, denoting $\overline\zeta^{\rm in}_k=\frac{d^k\overline \zeta}{dt^k}(0)$, we use the integro-differential equation \eqref{added-springmass-0} to write 
\begin{equation}\label{overzeta-k-in}
	\overline\zeta_{k+2}^{\rm in}= P_k^2(\overline\zeta^{\rm in}, \overline\zeta_1^{\rm in}, \zeta^{\rm in}(-r),\dots, \zeta_k^{\rm in}(-r)), \quad k\geq 0,
\end{equation}
where $P^2_k$ is a nonlinear function of its arguments.

\begin{definition}\label{def-compatibility}
	Let $m \geq 1$ be an integer. We say that the initial data $(\zeta^{\rm in}, q^{\rm in})\in H^m(\mathcal{E_-})$ and $(\overline\zeta^{\rm in}, \overline\zeta_1^{\rm in})\in \mathbb{R}^2$ for the initial boundary value problem \eqref{shallow-open}-\eqref{added-springmass-0} and \eqref{initial-conds} satisfy the compatibility conditions up to order $m-1$ if
	\begin{equation}
		q_k^{\rm in} = \overline{\zeta}_{k+1}^{\rm in}|\mathcal{E}_+| \quad \text{at}\quad x=-r, \qquad \text{for} \quad  k=0,\dots, m-1,
	\end{equation}
	with $q_k^{\rm in}$ and $\overline\zeta^{\rm in}_{k+1}$ as in \eqref{q-k-in} and \eqref{overzeta-k-in}.
\end{definition}

\begin{theorem}\label{theo-wave-spring-0}
	Let $m\geq 2$ be an integer. Consider $(\zeta^{\rm in}, q^{\rm in})\in H^m (\mathcal{E}_-)$ and $(\overline\zeta^{\rm in},\overline\zeta_1^{\rm in})\in \mathbb{R}^2$ satisfying 
	\begin{equation}\label{subsonic-ini}
		\inf_{\mathcal{E}_-}\Big(\sqrt{1+\eps \zeta^{\rm in}} - \eps\Big|\frac{q^{\rm in}}{1+\eps \zeta^{\rm in}}\Big| \Big) >0 \quad \text{and} \quad h^{\rm in}_{\rm wc}=1+ \eps\eta_{\rm eq} + \eps\overline\zeta^{\rm in} > 0
	\end{equation}
	and the compatibility conditions up to order $m-1$ in the sense of Definition \ref{def-compatibility}. Then, for any $\eps\in (0,1]$ and $\widetilde{\mu}=0$, there exist $T>0$ and a unique solution $(\zeta, q, \overline\zeta)$ to \eqref{shallow-open}-\eqref{added-springmass-0}, with $(\zeta,q)\in\mathbb{W}_T^m(\mathcal{E}_-)$ and $\overline\zeta\in H^{m+2}(0,T)$, satisfying the initial conditions \eqref{initial-conds}.
\end{theorem}

\begin{proof}
	After introducing the unknowns $u=(\zeta, q)$, $G=(\overline\zeta, \frac{d\overline\zeta}{dt}, P_{\rm ch})$ and the initial data $u^{\rm in}=(\zeta^{\rm in}, q^{\rm in})$, $G^{\rm in}=(\overline\zeta^{\rm in}, \overline\zeta^{\rm in}_1, P^{\rm in}_{\rm ch})$ with $P^{\rm in}_{\rm ch}= \epsilon \overline\zeta^{\rm in}$, the initial boundary value problem related to the wave-spring-mass system \eqref{shallow-open}-\eqref{added-springmass-0} can be recast as the $2\times2$ quasilinear hyperbolic system
	\begin{equation}\label{quasi-system}
		\begin{cases}
			\partial_t u + A(u)\partial_x u =0\quad &\text{in} \quad \mathcal{E}_-\\
			e_2\cdot u  = V(G)\quad &\text{at}\quad x=-r\\
			u=u^{\rm in} \quad &\text{at}\quad t=0
		\end{cases}
	\end{equation}
	with 
	$$A(u)=\left( \begin{matrix}
		0 & 1\\[2pt]
		h -\dfrac{\eps^2q^2}{h^2}& \dfrac{2\eps q}{h}
	\end{matrix}\right) \qquad \text{and} \qquad V(G)= |\mathcal{E}_+| G_2,$$
	coupled with the ODE
	\begin{equation}\label{semilinear-ODE}\begin{cases}
			\dfrac{dG}{dt}= M (G, e_1\cdot u (-r)),\\[5pt]
			G(0)=G^{\rm in},
		\end{cases}
	\end{equation}
	where $M: \mathbb{R}^3\times \mathbb{R} \rightarrow \mathbb{R}^3$ has components 
	\begin{align*}
		&M_1(G, \psi)= G_2, \ M_2(G, \psi) = \frac{1}{\mathcal{T}_{\rm buo}^2(h_{\rm wc}(G_1))} \Big(\!- G_1 - \tfrac{P_0}{\epsilon}G_3 +\eps \beta\left(h(\psi), h_{\rm wc}(G_1)\right)G_2^2  + \psi \Big),\\ 
		& M_3(G, \psi)=-\kappa G_3 + \epsilon G_2.
	\end{align*}
	The boundary condition in \eqref{quasi-system} is semilinear in the sense that $M$ in \eqref{semilinear-ODE} is nonlinear only with respect to $u$ and not its derivatives. Moreover, it follows from the definition of $\mathcal{T}^2_{\rm buo}(h_{\rm wc}(G_1))$ that $M_2(G, \psi)$ is well-defined whenever $h_{\rm wc}(G_1)>0$. This coupled PDE-ODE problem falls within the framework of \cite[Section 4]{IguLan21}.
The eigenvalues of the matrix $A(u)$ and the associated unit eigenfunctions are respectively given by
	$$\lambda_\pm(u)= \frac{\eps q}{h} \pm \sqrt{h} \qquad \text{and}\qquad  e_\pm (u)= \frac{(1, \lambda_\pm(u))}{\sqrt{1+ \lambda^2_\pm(u)}}.$$
The condition ensuring that $\lambda_+(u)>0$ and $\lambda_-(u)<0$, as well as the uniform Kreiss-Lopatinskii condition $|e_2\cdot e_+(u)|>0$, is 
	$$\sqrt{h(t,x)} -\eps \left|\frac{q}{h}(t,x)\right|>0 \qquad \forall (t,x)\in (0,T)\times \mathcal{E}_-,$$ 
	which is referred to as the subsonic regime in the context of nonlinear shallow water equations.
	Therefore, under the assumption \eqref{subsonic-ini} and the compatibility conditions introduced in Definition \ref{def-compatibility}, we argue as in the proof of \cite[Theorem 4.4]{IguLan21} and obtain the existence of $T>0$ and a unique solution $(u, G)$ to \eqref{shallow-open}-\eqref{added-springmass-0} with $u\in W^m_T(\mathcal{E}_-)$ and $G\in H^{m+1}(0,T)$. Consequently, $\overline\zeta \in H^{m+2}(0,T)$ and, using the third component of the ODE in \eqref{semilinear-ODE}, that $P_{\rm ch}\in H^{m+2}(0,T)$.	
\end{proof}

\emph{Dispersive case.} In the case $\widetilde{\mu}\neq 0$, the wave-spring-mass system \eqref{shallow-open}-\eqref{coupling-con} and \eqref{added-springmass} can be reformulated as an ODE and local well-posedness follows from Cauchy-Lipschitz theory, analogously to Theorem \ref{theo-trans-wp-disp}. As in that case, the existence time obtained here is not uniform with respect to $\widetilde{\mu}$, see Remark \ref{rem-uniform}.

\begin{theorem}\label{theo-wave-spring}
	Let $m\geq 1$ be an integer. Consider $(\zeta^{\rm in}, q^{\rm in})\in H^m(\mathcal{E}_-)\times H^{m+1}(\mathcal{E}_-) $, $(\overline\zeta^{\rm in }, \overline\zeta_1^{\rm in })\in \mathbb{R}^2$ such that 
	\begin{equation}\label{initial-ass-wc}
		\inf_{ \mathcal{E}_-}\left(1+\eps \zeta^{\rm in}\right) >0,  \quad h^{\rm in}_{\rm wc}=1+ \eps\eta_{\rm eq} + \eps\overline\zeta^{\rm in} > 0 \quad\text{and}\quad q^{\rm in}= |\mathcal{E}_+|\overline\zeta_1^{\rm in} \ \text{ at }\ x=-r.  
	\end{equation}
	Then, for any $\eps\in (0,1]$ and $\widetilde{\mu}>0$, there exists $T>0$ such that the system \eqref{shallow-open}-\eqref{coupling-con} and \eqref{added-springmass} admits a unique solution $(\zeta,q, \overline\zeta)\in C^1\left([0,T);H^m(\mathcal{E}_-)\times H^{m+1}(\mathcal{E}_-)\right)\times C^2([0,T); \mathbb{R})$ satisfying the initial conditions \eqref{initial-conds}.
	In addition, denoting by $T_{\rm max}$ the maximal existence time, if $T_{\rm max}$ is finite then 
	\begin{equation*}
		\limsup_{t\rightarrow T_{\rm max}^-}\Big( \Big\| \zeta(t), q(t), \frac{1}{h(t)}\Big\|_{L^\infty(\mathcal{E}_-)}  + \Big|\frac{d\overline\zeta}{dt}(t)\Big| + \Big| \frac{1}{h_{\rm wc}(\overline\zeta(t))} \Big|  \Big) = +\infty.
	\end{equation*} 
\end{theorem}

\begin{proof}
	We start by writing the coupled problem as an infinite-dimensional ODE.
	Using \eqref{explicit-timederq} and since $\mathcal{T}^2_{\widetilde{\mu}, \rm buo}(h, h_{\rm wc})$ is well-defined and positive whenever $h(-r)$ and $h_{\rm wc}$ do not vanish, we write system \eqref{shallow-open} and \eqref{added-springmass} as 
	\begin{equation}\label{ODE-0-wc}
		\begin{cases}	
			\partial_t \zeta = -\partial_x q \qquad &\text{in} \quad \mathcal{E}_-\\[5pt]
			\partial_t q = -R_0\partial_x \varphi + |\mathcal{E}_+|\dfrac{d^2\overline\zeta}{dt^2} e^{\frac{x+r}{\widetilde{\mu}}} \qquad &\text{in} \quad \mathcal{E}_-\\[7pt]
			\dfrac{d^2\overline\zeta}{dt^2}= \dfrac{1}{\mathcal{T}^2_{\widetilde{\mu}, \rm buo}(h, h_{\rm wc})} \Big (\!-\overline\zeta -\dfrac{P_0}{\epsilon}P_{\rm ch} + \eps \widetilde{\beta}(h, h_{\rm wc})\left(\dfrac{d\overline\zeta}{dt}\right)^2 + \mathcal{F}(-r)\Big)
			\\[12pt]
			\dfrac{dP_{\rm ch}}{dt}  =- \kappa P_{\rm ch} +{\epsilon}\dfrac{d\overline\zeta}{dt}
		\end{cases}
	\end{equation}
	We now denote by $W= (W_1, W_2, W_3, W_4, W_5)$ and $W^{\rm in}$ respectively the quintuples of unknowns $(\zeta, q, \overline\zeta, \frac{d\overline\zeta}{dt}, P_{\rm ch})$ and initial data $(\zeta^{\rm in}, q^{\rm in},\overline\zeta^{\rm in}, \overline\zeta_1^{\rm in}, P_{\rm ch}^{\rm in})$, with $P_{\rm ch}^{\rm in}$ as in \eqref{semilinear-ODE}. Then, after plugging the third line of \eqref{ODE-0-wc} into the second one, we write the Cauchy problem related to \eqref{ODE-0-wc} in the compact form
	\begin{equation}\label{compact-ODE-wc}
		\begin{cases}
			\dfrac{d W}{dt}= \Psi (W),\\[5pt]
			W(0)= W^{\rm in},
		\end{cases}
	\end{equation} 
	with 
	\begin{align*}
		&	\Psi_1(W)= -\partial_x W_2, \qquad \Psi_2(W)= -R_0\partial_x \left[ \varphi(W_1,W_2) \right]+ |\mathcal{E}_+|\Psi_4(W) e^{\frac{x+r}{\widetilde{\mu}}}, \qquad \Psi_3(W)= W_4,\\[2pt]
		&	\Psi_4(W) = \dfrac{1}{\mathcal{T}^2_{\widetilde{\mu}, \rm buo}\left(h(W_1), h_{\rm wc}(W_3)\right)} \Big[\!\!-W_3\!-\!\frac{P_0}{\epsilon}W_5 + \eps \widetilde{\beta}\left(h(W_1), h_{\rm wc}(W_3)\right)W_4^2 \\&\qquad \qquad \qquad \qquad \qquad \qquad\qquad \quad\quad +\!\left(\mathcal{F}(W_1,W_2)\right)\!(-r)\Big]\!,\\[2pt]
		&	\Psi_5(W)  =- \kappa W_5 +{\epsilon}W_4.
	\end{align*}
	
	Let us denote by $\mathcal{N}$ the open subset of $H^m(\mathcal{E}_-)\times H^{m+1}(\mathcal{E}_-)\times \mathbb{R}^3$ with $\inf_{\mathcal{E}_-} h>0$ and $h_{\rm wc}>0$. Analogously to the proof of Theorem \ref{theo-trans-wp-disp}, here $\Psi$ is a smooth map from $\mathcal{N}$ to $H^m(\mathcal{E}_-)\times H^{m+1}(\mathcal{E}_-)\times \mathbb{R}^3$ for $m\geq 1$. We then apply the Cauchy-Lipschitz theorem to establish the existence of $T>0$ and a unique solution $W\in C^1([0,T); H^m(\mathcal{E}_-)\times H^{m+1}(\mathcal{E}_-)\times \mathbb{R}^3)$ to \eqref{compact-ODE-wc}, which in turn implies that $\overline\zeta \in C^2([0,T); \mathbb{R})$. In addition, using the fifth component of the ODE in \eqref{compact-ODE-wc}, we also have that $P_{\rm ch}\in C^2([0, T); \mathbb{R})$. 
	Due to the assumption on $q^{\rm in}$ and $\overline\zeta^{\rm in}_1$ in \eqref{initial-ass-wc}, we have the equivalence 
	$$q=|\mathcal{E}_+|\dfrac{d\overline\zeta}{dt} \quad \text{at} \quad x=-r \quad \iff\quad \partial_t q=|\mathcal{E}_+|\dfrac{d^2\overline\zeta}{dt^2} \quad \text{at} \quad x=-r $$ 
	and the regularity of the solution together with the second equation in \eqref{ODE-0-wc} guarantees that the boundary condition in \eqref{coupling-con} is satisfied on $[0,T)$. Concerning the blow-up criterion, we argue as in the proof of Theorem \ref{theo-trans-wp-disp} using that Moser-type estimates imply that 
	\begin{align*}
		&	\|\Psi(W)\|_{H^m(\mathcal{E}_-)\times H^{m+1}(\mathcal{E}_-)\times \mathbb{R}^3} \\[2pt]&\leq C_{\widetilde{\mu}} \Big( \Big\| W_1, W_2, \frac{1}{h(W_1)}\Big\|_{L^\infty(\mathcal{E}_-)} , \left|W_4\right| , \Big| \frac{1}{h_{\rm wc}(W_3)} \Big| \Big) \|W\|_{H^m(\mathcal{E}_-)\times H^{m+1}(\mathcal{E}_-)\times \mathbb{R}^3}.
	\end{align*}\end{proof}

\appendix

\section{Non-dimensionalization}\label{sec-appendix}
In this appendix, we show the details of the derivation of some dimensionless equations and elastic energy introduced throughout the paper.
\subsection{Air pressure dynamics}\label{appsec-air}
Let us write the evolution equation \eqref{Pch-dyn} of the air pressure inside the OWC chamber as
\begin{equation}\label{Pch-dyn-ex}
	\frac{d P_{\rm ch}}{dt} +k P_{\rm ch} = 	\frac{\gamma P_{\rm atm }}{h_{\rm ch}} \frac{d}{dt}\Big(\frac{1}{|\mathcal{E}_+|}\int_{\mathcal{E}_+} \zeta\Big)
\end{equation}coupled with the initial condition
\begin{equation}\label{limit-Pch}
	P_{\rm ch}(0)= \frac{\gamma P_{\rm atm}}{h_{\rm ch}}\overline\zeta (0).
\end{equation}
We introduce the dimensionless spatial and time variables 
\begin{equation*}
	x'= \frac{x}{L}, \qquad z'=\frac{z}{h_0}, \qquad t'=\frac{t}{L/\sqrt{gh_0}}
\end{equation*}
and the dimensionless air pressure deviation, surface elevation and damping constant 
\begin{equation}\label{air-sur-dimless}
	P'_{\rm ch}= \frac{P_{\rm ch}}{\gamma P_{\rm atm}}, \qquad \zeta' = \frac{\zeta}{a}, \qquad \kappa=\frac{k L}{\sqrt{gh_0}}.
\end{equation}Injecting the dimensionless quantities into \eqref{Pch-dyn-ex}-\eqref{limit-Pch} then yields 
\begin{equation*}\label{eq-Pch-dimless}
	\frac{d P'_{\rm ch}}{dt'} +\kappa P'_{\rm ch} = \epsilon \frac{d}{dt'}\Big(\frac{1}{|\mathcal{E}'_+|}\int_{\mathcal{E}'_+} \zeta'\Big), \qquad P'_{\rm ch}(0)= \epsilon \overline \zeta (0),
\end{equation*}
where $\mathcal{E}'_+= \left(\frac{r}{L}, \frac{\ell}{L}\right)$ is the dimensionless chamber domain, which admits the unique solution 
\begin{equation}\label{Pch-dimless1}
	P'_{\rm ch}=\epsilon(I - \kappa \mathcal{D})\overline{\zeta'},
\end{equation} where $I$ is the identity operator and $\mathcal{D}$ is the delay operator defined by
\begin{equation*}
	\mathcal{D}\overline{\zeta'}(t')= \int_{0 }^{t'} e^{-\kappa (t'-s)}\overline{\zeta'}(s) ds.
\end{equation*}
Moreover, after introducing the dimensionless surface pressure
\begin{equation}\label{surpres-dimless}
	\underline{P}'= \frac{\underline{P}}{\rho g h_0},
\end{equation} the surface pressure constraint \eqref{surpres-ext} in the chamber domain becomes
\begin{equation*}
	\underline{P}'= \frac{\gamma P_{\rm atm}}{\rho g h_0} P'_{\rm ch} \qquad \text{in} \quad \mathcal{E}'_+.
\end{equation*}

\subsection{Elastic potential energy}\label{appsec-elastic}
Associated with Hooke's law \eqref{spring-force}, we define the elastic potential energy 
\begin{equation}\label{elastic-energy}
	E_{\rm ela}(\zeta)=\frac{1}{2} K \overline{\zeta}^2= \frac{1}{2}\frac{\gamma P_{\rm atm}|\mathcal{E}|}{h_{\rm ch}} \overline{\zeta}^2.
\end{equation}
After introducing the dimensionless elastic potential energy 
\begin{equation*}
	E'_{\rm spring}= \frac{E_{\rm ela}}{\rho g La^2 },
\end{equation*} we obtain from \eqref{elastic-energy} that
\begin{equation*}
	\rho g La^2 	E'_{\rm spring}(\zeta')= \frac{1}{2} Ka^2\overline{\zeta'}^2= \frac{1}{2} \frac{\gamma P_{\rm atm }|\mathcal{E}'|L}{ h_{\rm ch} }a^2\overline{\zeta'}^2, 
\end{equation*}which yields the dimensionless elastic potential energy definition
\begin{equation*}
	E_{\rm ela}'(\zeta')=\frac{1}{2}\tau \overline{\zeta'}^2
\end{equation*} with dimensionless stiffness $\tau=\dfrac{K}{\rho g L}= \dfrac{\gamma P_{\rm atm}|\mathcal{E}'|}{\rho g h_{\rm ch}}$.

\subsection{Newton's second law}\label{appsec-newton}
The motion of the water column in the OWC chamber is determined by Newton's second law \eqref{newton0}.
After introducing the dimensionless water column mass and interface pressure 
\begin{equation}
	m'_{\rm wc}= \dfrac{m_{\rm wc}}{ \rho h_0|\mathcal{E}_+|}, \qquad  \underline{\mathcal{P}}'=\dfrac{\underline{\mathcal{P}}}{\rho g h_0},
\end{equation} we rewrite \eqref{newton0} as
\begin{equation}\label{newton0-dimless}
	m'_{\rm wc}\frac{h_0^2}{L^2}\frac{d^2\overline{\zeta'}}{dt'^2}= -\frac{m'_{\rm wc}}{\eps} - \frac{\gamma P_{\rm atm}}{\rho g a}P'_{\rm ch}+\frac{1}{\eps |\mathcal{E}'_+|}\int_{\mathcal{E}'_+}\underline{\mathcal{P}}',
\end{equation}with $\zeta'$, $P'_{\rm ch}$ defined in \eqref{air-sur-dimless} and \eqref{surpres-dimless}. After resorting to \eqref{Pch-dimless1}, we end up with 
\begin{equation}\label{newton1-dimless}
	m'_{\rm wc}\frac{h_0^2}{L^2}\frac{d^2\overline{\zeta'}}{dt'^2}= -\frac{m'_{\rm wc}}{\eps} - P_0 (I-\kappa \mathcal{D})\overline{\zeta'}+\frac{1}{\eps |\mathcal{E}'_+|}\int_{\mathcal{E}'_+}\underline{\mathcal{P}}'
\end{equation}with $P_0$ as in \eqref{pressure-mean}. An equivalent form is obtained by means of Archimedes' principle. We recall that 
\begin{equation}\label{archi-prin}
	m_{\rm wc}= -\rho|\mathcal{E}_+|\eta_{\rm eq}
\end{equation}where $\eta_{\rm eq}<0$ is the dimensional vertical position of the water column bottom in an equilibrium state. Introducing the dimensionless vertical position of the water column bottom
\begin{equation*}
	\eta'_{\rm eq}=\frac{\eta_{\rm eq}}{a},
\end{equation*}
the dimensionless version of \eqref{archi-prin} reads
\begin{equation}\label{archimede-dimless}
	m'_{\rm wc}= -\eps \eta'_{\rm eq}.
\end{equation}From the dimensional interface contact constraint
\begin{equation}
	\eta=\eta_{\rm eq}+\overline{\zeta} \quad \text{in} \quad \mathcal{E_+},
\end{equation}we derive the dimensionless version
\begin{equation}\label{inter-con-dimless}
	\eta'= \eta'_{\rm eq}+\overline{\zeta'} \quad \text{in} \quad \mathcal{E}'_+,
\end{equation}where $\eta'={\eta}/{a}$
is the dimensionless interface elevation. Combining \eqref{archimede-dimless} with \eqref{inter-con-dimless}
allows us to rewrite \eqref{newton1-dimless} in the equivalent form
\begin{equation}\label{equi-newton-dimless}
	m'_{\rm wc}\frac{h_0^2}{L^2} \frac{d^2 \overline{\zeta'}}{dt'^2} = -\overline{\zeta'} -P_0(I-\kappa \mathcal{D})\overline{\zeta'} +\eta'+\frac{1}{\eps |\mathcal{E}'_+|}\int_{\mathcal{E}'_+} \underline{\mathcal{P}}'.
\end{equation} This formulation naturally leads to the introduction of the dimensionless hydrostatic buoyancy period $2\pi T_{\rm buo}$, defined by
$$ T^2_{\rm buo}=\frac{m'_{\rm wc}h_0^2}{L^2}= \frac{m_{\rm wc}h_0}{\rho |\mathcal{E}_+|L^2},$$
and, using this notation, \eqref{equi-newton-dimless} becomes 
\begin{equation*}
	T^2_{\rm buo}\frac{d^2\overline{\zeta'}}{dt'^2}= -\overline{\zeta'} - P_0 (I-\kappa \mathcal{D})\overline{\zeta'}+ \eta'+\frac{1}{\eps |\mathcal{E}'_+|}\int_{\mathcal{E}'_+}\underline{\mathcal{P}}'.
\end{equation*}

\subsection*{Acknowledgements.} The author is supported by the grant \textit{Dipartimento di 
	Eccellenza 2023-2027},
issued by the Italian Ministry of University and Research (MUR). He is also partially supported by the Gruppo Nazionale
per l’Analisi Matematica, la Probabilità e le loro Applicazioni (GNAMPA) of the
Istituto Nazionale di Alta Matematica (INdAM). 

\addcontentsline{toc}{section}{References}
\bibliographystyle{abbrv}
\bibliography{references}

\end{document}